\documentclass[a4paper, 11pt]{article}
\usepackage{amssymb}
\usepackage{amsthm}
\usepackage{amsmath}
\usepackage{graphicx}
\usepackage{amsthm}
\usepackage{tikz}
\usetikzlibrary{arrows,shapes,trees,decorations.markings}
\usetikzlibrary{graphs,graphs.standard}

\newtheorem{thm}{Theorem}[section]
\newtheorem{cor}[thm]{Corollary}

\newtheorem{prop}[thm]{Proposition}
\newtheorem{conj}[thm]{Conjecture}

\newtheorem{notation}[thm]{Notation}
\theoremstyle{definition}
\newtheorem{defn}[thm]{Definition}
\theoremstyle{remark}

\numberwithin{equation}{section}

\def\N{\mathbb{N}}

\newcommand{\F}{\mathbb{F}}
\newcommand{\K}{\mathbb{K}}

\newcommand{\p}{\mathbb{P}}

\newcommand{\Z}{\mathbb{Z}}

\def\Fq{\mathbb{F}_q}

\begin{document}

\title{Graph decompositions in projective geometries}

\author{Marco Buratti \thanks{Dipartimento di Matematica e Informatica, Universit\`a di Perugia, via Vanvitelli 1 - 06123 Italy, email: buratti@dmi.unipg.it}
\quad\quad
Anamari Naki\'c \thanks{Faculty of Electrical Engineering and Computing,
University at Zagreb, Croatia, email: anamari.nakic@fer.hr}
\quad\quad
Alfred Wassermann \thanks{Department of Mathematics, University of Bayreuth, D-95440 Bayreuth,
Germany, email: alfred.wassermann@uni-bayreuth.de}}


\maketitle
\begin{abstract}
Let PG$(\F_q^v)$ be the $(v-1)$-dimensional projective space over $\F_q$ and let $\Gamma$ be a simple graph of order ${q^k-1\over q-1}$ for some $k$.
A 2$-(v,\Gamma,\lambda)$ design over $\F_q$ is a collection $\cal B$ of graphs ({\it blocks}) isomorphic to $\Gamma$ with the following properties:
the vertex set of every block is a subspace of PG$(\F_q^v)$; every two distinct points of PG$(\F_q^v)$ are adjacent in exactly $\lambda$ blocks. This new definition covers, in particular, the well known concept of a 2$-(v,k,\lambda)$ design over $\F_q$ corresponding to the case that $\Gamma$ is complete. 

In this work of a foundational nature we illustrate how difference methods allow us to
get concrete non-trivial examples of $\Gamma$-decompositions over $\F_2$ or $\F_3$ for which $\Gamma$ is a cycle, a path, a prism, a generalized Petersen graph, or a Moebius ladder. 
In particular, we will discuss in detail the special and very hard case that $\Gamma$ is complete and $\lambda=1$, i.e.,
the Steiner 2-designs over a finite field. Also, we briefly touch the new topic of near resolvable 2-$(v,2,1)$ designs over $\F_q$.

This study has led us to some (probably new) collateral problems concerning difference sets. 
Supported by multiple examples, we conjecture the existence of infinite families of $\Gamma$-decompositions over a finite field
that can be obtained by suitably labeling the vertices of $\Gamma$ with the elements of a Singer difference set.
\end{abstract}

\small\noindent {\bf Keywords:} design over a finite field; group divisible design over a finite field;
projective space; spread;
graph decomposition; difference set; difference family.

\medskip\noindent {\bf Mathematics Subject Classification (2010):} 05B05, 05B10, 05B25, 05B30. 
\eject
\normalsize
\section{Introduction}
This work has been inspired by the natural relationship between classic 2-designs and 2-designs over a finite field, and between classic 2-designs and graph decompositions. 
Hence, it involves three of the main characters of combinatorics that are designs, graphs, and finite geometries.

Designs over finite fields have been introduced in the 1970's \cite{CI74, CII74, D76}. Even though they are
a generalization of the classic $t$-designs that date back to the 1930's, they are not $t$-designs in the classic sense
when $t>2$.
Instead, a $2$-design over a finite field $\F_q$ is a $2$-design in the classic sense
with the tremendous constraint that its points are those of a projective space $\p$ over $\F_q$ and that its blocks
are suitable subspaces of $\p$. 
This is the reason for which in this paper we are interested only in the special case $t=2$. 

The theory of graph decompositions originates from design theory, especially from the  quite evident fact 
that a $2$$-(v, k,1)$ design is completely equivalent to  a decomposition of the complete graph of order $v$ into cliques of order $k$. 
Thus, for what said above, a $2$-design over $\F_q$ can be seen as a decomposition
of the complete graph whose vertices are the points of a projective space $\p$ over $\F_q$ into cliques each of which has 
a subspace of $\p$ as vertex set. 

Based on the above observations, we propose to consider, much more generally, decompositions of a graph $\K$ with vertex set the
points of a projective space $\p$ into copies of a graph $\Gamma$ each of which has a subspace of $\p$ as vertex set.
If $\p$ is over $\F_q$ and its dimension is $v-1$, we will refer to this decomposition as a $2$$-(v,\Gamma,1)$ design over $\F_q$.

Designs over finite fields have recently received a huge amount of attention because of the discovering in \cite{BEO16} of a non-trivial Steiner 2-design over a finite field
which was conjectured could not exist for a long time. Unfortunately, for the time being, it seems that there is no hope to find others.
To relax the definition of a Steiner 2-design over $\F_q$ to that of a $2$$-(v,\Gamma,1)$ design over $\F_q$ seems to promise much more 
results even though the related problems still appear very difficult, hence challenging. 
Thus we hope that this new topic of ``graph decompositions over a finite field" may open the doors to a fruitful research in the future.

The article will be structured as follows.
The next section is useful for understanding the notation and terminology used throughout the paper.
In Section 3 we recall the definitions of a classic 2-design, of a group divisible design, and their versions over a finite field.
All these definitions will be relaxed in Section 4 in terms of graph decompositions.
In Section 5, generalizing what  has already been done by the first two authors in \cite{BN},
the difference methods of the classic design theory will be adapted to graph decompositions over a finite field; these
methods will be crucial for all the concrete constructions that we are able to exhibit in the paper.
In Section 6 we discuss the possible existence of a Steiner 2-design over a finite field. In particular, we prove a
necessary condition which seems to have gone unnoticed before and we revisit the non-trivial Steiner 2-design over a finite field mentioned above. 
In Sections 7 and 8 we give necessary conditions for the existence of cycle- and path-decompositions over
a finite field providing some concrete examples over $\F_2$ and $\F_3$. In Section 9 we propose some problems on difference sets. 
In particular, given a $(2s+1,k,\lambda)$ difference set $D$ in a group $G$, and given a graph $\Gamma$ of order $k$ and size $s$, 
we ask whether it is possibile to label the vertices of $\Gamma$ with the elements of $D$ is such a way that every non-identity element 
of $G$ may be expressed as a difference of two 
``adjacent labels". We conjecture that the answer is always affirmative if $\Gamma$ is regular and connected.
The proof of this conjecture when $D$ is a Singer difference set would give an infinite family of non-trivial decompositions 
over a finite field, i.e., a $2$$-(v,\Gamma,1)$ design over $\F_q$ for every regular and connected graph $\Gamma$ of order ${q^{v-1}-1\over q-1}$.
In Section 10 such a design has been found in each of the following cases: $v\in\{4,5,6,7\}$, $q=2$, and $\Gamma$ is the $(2^{v-1}-1)$-cycle;
$v=5$, $q=3$, and $\Gamma$ is a prism or a generalized Petersen graph or a Moebius ladder on 40 vertices.
Finally, in the last section, we make a short discussion about {\it improper} graph decompositions over a finite field, i.e.,
$2$$-(v,\Gamma,1)$ designs over $\F_q$ where $\Gamma$ has at least one isolated vertex.

\section{Notation and terminology}
For $q$ a prime power, $\F_q$ will denote the finite field of order $q$ and PG$(\Fq^v)$ is the $(v-1)$-dimensional projective space over $\F_q$.
The number of points of PG$(\Fq^v)$ will be denoted by $[v]_q$. Hence we have: $$[v]_q = \frac{q^v-1}{q-1}=\sum_{i=0}^{v-1}q^i$$
By $[\Z_v]_q$ we will denote the {\it Singer group} of order $[v]_q$, that is the quotient group between the
multiplicative groups of the fields of order $q^v$ and $q$: $$[\Z_v]_q=\F_{q^v}^*/\F_q^*.$$
This group acts sharply transitively on the point-set of PG$(\F_q^v)$. Hence,
the points of PG$(\F_q^v)$ will be always identified with the elements of $[\Z_v]_q$.

By analogy, given that $m\Z_{mn}=\{mi \ | \ 0\leq i\leq n-1\}$ is the subgroup of $\Z_{mn}$ of order $n$,
the subgroup of $[\Z_{mn}]_q$ of order $[n]_q$ will be denoted by $[m\Z_{mn}]_q$.
Thus, if $g$ is a generator of $\F_{q^{mn}}^*$, we have:
$$[m\Z_{mn}]_q=\{g^{i[m]_{q^n}} \ | \ 0\leq i\leq [n]_q-1\}.$$

Throughout the paper, $\K_v$ is the complete graph on an abstract set of $v$ vertices, 
$\K_V$ is the complete graph on a concrete set $V$, and $\K_{m\times n}$ is the complete $m$-partite graph 
whose parts have size $n$. The cycle and the path on $k$-vertices will be denoted by $C_k$ and $P_k$, respectively.  If $\lambda$ is a positive integer and $\Gamma$ is a graph, then $\lambda\Gamma$
will denote the $\lambda$-fold of $\Gamma$.

 If $\Gamma$ is an abstract graph, when we refer to a $\Gamma$-subgraph of $\K_V$ we will mean a subgraph of $\K_V$ isomorphic to $\Gamma$.
If $q$ is a prime power, an abstract graph $\Gamma$ will be called {\it $q$-spaceable} if its order is the number of points
of a projective space over $\F_q$ of a suitable dimension, i.e., if its vertex set has size $[k]_q$ for some $k$.
Finally, speaking of a $\Gamma$-subspace of PG$(\F_q^v)$ we will mean a graph isomorphic to $\Gamma$
whose vertex set is a subspace of PG$(\F_q^v)$. Of course, in this case $\Gamma$ must be
$q$-spaceable.

Let $G$ be a group and let $B$ be a simple graph with vertices in $G$. By {\it list of differences} of $B$ we will mean 
the multiset $\Delta B$ of all possible differences $x-y$ or quotients $xy^{-1}$
(depending on whether $G$ is additive or multiplicative) with $(x,y)$ an ordered pair of adjacent vertices of $B$. Note that if $B$ is complete with vertex set $V(B)$, then
$\Delta B$ coincides with the list of differences of the set $V(B)$ in the usual sense. 

The list of differences of $B$ can be conveniently displayed by means of its {\it difference table}. This is the square matrix $T(B)$
whose rows and columns are labeled with the vertices $b_1$, \dots, $b_k$ of $B$ and where the entry $t_{ij}$ is empty
or equal to the difference $b_i-b_j$ (or the quotient $b_ib_j^{-1}$ if $G$ is multiplicative) according to whether $b_i$ is not adjacent or adjacent to $b_j$, respectively. 

If $\cal F$ is a family of subgraphs of $\K_G$, then the
list of differences of ${\cal F}$ is the multiset sum $\displaystyle \Delta{\cal F}=\biguplus_{B\in {\cal F}}\Delta(B)$.

The {\it development} of $B$ and $\cal F$, denoted by dev$B$ and dev${\cal F}$, are the multisets of graphs defined by
$${\rm dev}(B)=\{B_g \ | \ g\in G\}\quad\quad{\rm and}\quad\quad {\rm dev}{\cal F}=\biguplus_{B\in {\cal F}}{\rm dev}(B)$$
where $B_g$ is the graph obtained from $B$ by replacing each $b\in V(B)$ with $b+g$ or $bg$ according
to whether $G$ is additive or multiplicative, respectively.

\section{Designs over a finite field}

Let us recall the well known notion of a $t$-design.
\begin{defn}
A $t$$-(v, k,\lambda)$ design is a pair $({\cal P},{\cal B})$ where $\cal P$ is a set of $v$ {\it points}, 
and $\cal B$ is a collection of $k$-subsets ({\it blocks}) of $\cal P$ such that every $t$-subset of $\cal P$ is contained in exactly $\lambda$ blocks.
\end{defn}

The above definition has been generalized as follows.
\begin{defn}
A $t$$-(v,k,\lambda)$ {\it design over $\F_q$} -- or a $t$$-(v,k,\lambda)_q$ design to be brief -- is a collection $\cal S$ of $k$-dimensional subspaces of the vector space 
$\F_q^v$ with the property that any $t$-dimensional subspace of $\F_q^v$ is contained in exactly $\lambda$ members of $\cal S$.
\end{defn}

A $t$$-(v,k,\lambda)$ design over $\F_q$ is also said the {\it $q$-analog} of a $t$$-(v,k,\lambda)$ design or
a $t$$-(v, k, \lambda)_q$ {\it subspace design}.

In this paper we are interested only in the special case that $t=2$.
Every $q$-analog of a $2$$-$design can be seen as a 2$-$design in the classic sense. Indeed, if PG$(\Fq^v)$ is the $(v-1)$-dimensional
projective space over $\F_q$,  the definition of a $2$$-(v,k,\lambda)_q$ design can be equivalently reformulated as follows.
\begin{defn}\label{design_q}
A $2$$-(v, k,\lambda)$ design over $\F_q$ is a classic $2$$-([v]_q,[k]_q,\lambda)$ design $({\cal P},{\cal B})$ where $\cal P$ is the set
of points of PG$(\F_q^v)$ and where every $B\in\cal B$ is a $(k-1)$-dimensional subspace of PG$(\F_q^v)$.
\end{defn}
The set of points and the set of all possible $(k-1)$-dimensional subspaces of PG$(\Fq^v)$ is the {\it complete} 2$-(v,k,\lambda)_q$ design where
$\displaystyle\lambda=\prod_{i=1}^{k-2}{q^{v-k+i}-1\over q^{k-i}-1}$.
In particular, for $k=2$, the set of points and the set of lines of PG$(\Fq^v)$ is a 2$-(v,2,1)_q$ design. 
For an overview of known results about $2$$-(v, k , \lambda)_q$ designs see \cite{BKW18}. 

Now we recall the definitions of a classic group divisible design and of a group divisible design over a finite field.

\begin{defn}
A $(mn,n,k,\lambda)$ {\it group divisible design} (briefly GDD) is a triple $({\cal P},{\cal G},{\cal B})$
where $\cal P$ is a set of $mn$ points, ${\cal G}$ is a partition of ${\cal P}$ into $m$ sets
({\it classes}) of size $n$, and ${\cal B}$ is a collection of $k$-subsets of ${\cal P}$ ({\it blocks}) such that each block meets each class
in at most one point and any two points belonging to different classes are contained in exactly $\lambda$ blocks.
\end{defn}

We also recall that GDDs are often useful to construct 2$-$designs in many ways. In particular, it is evident
that the existence of a $(mn,n,k,\lambda)$-GDD and of a 2$-(n,k,\lambda)$ design implies the existence 
of a 2$-(mn,k,\lambda)$ design.

The $q$-analog of a classic GDD has been recently introduced in \cite{BKKNW}.
Its definition requires the notion of a {\it $d$-spread} of PG$(\F_q^v)$, that
is a partition of the set of points of PG$(\F_q^v)$ into $d$-dimensional subspaces.
Such a $d$-spread exists if and only if $d+1$ is a divisor of $v$. In particular, the {\it Desarguesian\break $(n-1)$-spread} of PG$(\F_q^{mn})$
is the partition of $[\Z_{mn}]_q$ into the cosets of $[m\Z_{mn}]_q$ (see, e.g., \cite{LV}).

\begin{defn}\label{gdd_q}
A $(mn,n,k,\lambda)$-GDD over $\F_q$ -- or a $(mn,n,k,\lambda)_q$-GDD to be brief -- is a
$([mn]_q,[n]_q,[k]_q,\lambda)$-GDD where the points are
those of PG$(\F_q^{mn})$, the classes are the members of a $(n-1)$-spread of PG$(\F_q^{mn})$,
and the blocks are $(k-1)$-dimensional subspaces of PG$(\F_q^{mn})$.
\end{defn}

As a special case of the remark that we have done on classic GDDs we can say that
combining a $(mn,n,k,\lambda)_q$-GDD with a 2$-(n,k,\lambda)_q$ design one obtains a 2$-(mn,k,\lambda)_q$ design.

\section{Graph decompositions over a finite field}
Now we want to make a link between designs over finite fields and {\it graph decompositions}.

\begin{defn}
Let $\Gamma$ be a simple graph. A 2$-(v,\Gamma,\lambda)$ design is a pair $({\cal P},{\cal B})$ where ${\cal P}$ is a set of $v$ points and
where ${\cal B}$ is a collection of $\Gamma$-subgraphs ({\it blocks}) of $\K_{\cal P}$ such that any two distinct points are 
adjacent in exactly $\lambda$ blocks.
\end{defn}

In most of the literature (see, e.g., \cite{BE06}) a design as above is said to be a {\it $(\lambda\K_v,\Gamma)$-design} or a 
{\it $\Gamma$-decomposition of $\lambda\K_v$}. Indeed, to say that $({\cal P},{\cal B})$ is a 2$-(v,\Gamma,\lambda)$ design 
is equivalent to say that the edge sets of its blocks partition the edge multiset of $\lambda\K_{\cal P}$.
We changed a bit the formal definition just in order to keep notation and terminology of graph decompositions similar to those of classic 2$-$designs.

It is evident that a 2$-(v,\K_k,\lambda)$ design is nothing but a classic 2$-(v,k,\lambda)$ design.
Thus, by Definition \ref{design_q}, any $2$-$(v,k,\lambda)_q$ design can be equivalently interpreted as
a decomposition of the $\lambda$-fold of the complete graph on the points of PG$(\Fq^v)$ into a collection of $\K_{[k]_q}$-subspaces of PG$(\Fq^v)$. 
This leads to the following new notion of a {\it graph decomposition over a finite field}.

\begin{defn}\label{graphdesign_q}
A 2$-(v,\Gamma,\lambda)$ design over $\F_q$ is a 2$-([v]_q,\Gamma,\lambda)$ design $({\cal P},{\cal B})$ such that $\cal P$ is the set of points of PG$(\Fq^v)$, 
and each $B\in{\cal B}$ is a $\Gamma$-subspace of PG$(\F_q^v)$.
\end{defn}

Note that a ``2$-(v,\K_{[k]_q},\lambda)$ design over $\F_q$"
is essentially a ``2$-(v,k,\lambda)_q$ design". Consitently, when $\Gamma$ is a $q$-spaceable cycle or a $q$-spaceable path, 
we will adopt the following notation.

\begin{notation}\label{q-troubles}
{\rm Speaking of a ``$2$$-(v,C_k,\lambda)_q$ design" we will mean a\break
``$2$$-(v,C_{[k]_q},\lambda)$ design over $\F_q$".
Analogously, speaking of a ``$2$$-(v,P_k,\lambda)_q$ design" we will mean a 
``$2$$-(v,P_{[k]_q},\lambda)$ design over $\F_q$".}
\end{notation}

Let us say that a 2$-(v,\Gamma,\lambda)$ design $({\cal P},{\cal B})$ is {\it spanning} if $\Gamma$ has order $v$, hence if all its 
blocks are spanning subgraphs of $\K_{\cal P}$.
It is obvious that every spanning 2$-([v]_q,\Gamma,\lambda)$ design can be seen as a 2$-(v,\Gamma,\lambda)$ design over $\F_q$;
it is enough to rename the vertices of $\K_{[v]_q}$ with the points of PG$(\F_q^v)$. Thus, in the framework of designs over finite fields, 
the spanning designs will be considered trivial. In spite of this fact they could be helpful for the construction of some graph decompositions over a finite field
which are not trivial at all (see next Corollary \ref{notsotrivial}).


Let us see which are the obvious necessary conditions for the existence of a graph decomposition over a finite field.

\begin{prop}\label{admissible}
	The trivial necessary conditions for the existence of a $2$$-(v,\Gamma,\lambda)$ design over $\F_q$ are the following:
	\begin{itemize}
	\item[(i)] $\Gamma$ is $q$-spaceable;
	\item[(ii)] the size of $\Gamma$ is a divisor of ${1\over2}\lambda q[v]_q[v-1]_q$;
	\item[(iii)] the greatest common divisor of the degrees of the vertices of $\Gamma$ divides $\lambda q[v-1]_q$.
	\end{itemize}
\end{prop}
\begin{proof} Let ${\cal D}=({\cal P},{\cal B})$ be a 2$-(v,\Gamma,\lambda)$ design over $\F_q$.
By definition, the vertex set of every $B\in{\cal B}$ is a subspace of PG$(\F_q^v)$,
hence $\Gamma$ is $q$-spaceable.
The other conditions follow from the necessary conditions for the existence of a classic graph decomposition.
The size of $\Gamma$ must be a divisor of the size of the $\lambda$-fold of $\K_{\cal P}$, which is equal to $\lambda{[v]_q \choose 2}$,
and a trivial computation shows that ${[v]_q \choose 2}={1\over2}q[v]_q[v-1]_q$.
Finally, for $P\in{\cal P}$ and $B\in{\cal B}$, let $deg_B(P)$ be the degree of $P$ in the graph $B$. Then it is obvious that
the sum $\sum_{B\in{\cal B}} deg_B(P)$ is the degree of $P$ in $\lambda\K_{\cal P}$, that is $\lambda([v]_q-1)=\lambda q[v-1]_q$. Considering that each
block is isomorphic to $\Gamma$, it is clear that $deg_B(P)$ is a degree of a vertex of $\Gamma$ for each $B$, hence  
$\lambda q[v-1]_q$ is divisible by the greatest common divisor of the degrees of the vertices of $\Gamma$.
\end{proof}

Note that the third admissibility condition is empty  in the case that the greatest common divisor of the degrees of the vertices of $\Gamma$
is 1. In contrast, it is particularly important when $\Gamma$ is regular. Indeed in this case condition (iii) can be more conveniently reformulated as follows.
\begin{quote}
(iii') If $\Gamma$ is a regular graph of degree $d$, then $d$ must be a divisor of $\lambda q[v-1]_q$.
\end{quote}
For instance, a non-trivial 2$-(7,\Gamma,1)$ design over $\F_2$ with $\Gamma=(V,E)$ connected, may exist only 
when the order and the size of $\Gamma$ are as follows.
\begin{center}
\begin{tabular}{|l|c|r|c|r|c|r|c|r|c|r|c|r|}
\hline {$|V|$} & $|E|$    \\
\hline $7$ & $\bf7$, \ $9$, \ $\bf21$  \\
\hline $15$ & $21$, \ $63$  \\
\hline $31$ & $63$, \ $127$, \ $381$  \\
\hline $63$ & $\bf63$, \ $127$, \ $381$, \ $889$, \ $1143$ \\
\hline
\end{tabular} 
\end{center}
The size is in boldface only when $\Gamma$ might be regular, hence in the cases that $|V|$ divides $2|E|$.
Note that $(|V|,|E|)=(7,21)$ corresponds to the case that $\Gamma$ is complete,
that is equivalent to the 2-analog of a {\it Fano plane}, namely to a 2$-(7,3,1)_2$ design. There is a great deal of doubt on
the existence of such a design; indeed it has been proved that if it exists, then its full automorphism group has order at 
most two \cite{BKN16, KKW16,ST87}. The case $(|V|,|E|)=(7,7)$ corresponds to a 2$-(7,C_3,1)_2$ design that will be constructed in
Section \ref{cyclesection}.
The case $(|V|,|E|)=(63,63)$ corresponds to a 2$-(7,C_6,1)_2$ design that will be constructed in 
Subsection \ref{singergracefulcycles}.

Now note that a $(mn,n,k,\lambda)$-GDD is equivalent to a $\K_k$-decomposition of $\lambda\K_{m\times n}$. 
Thus, in particular, a $(mn,n,k,\lambda)_q$-GDD can be seen as a decomposition of the
$[m]_{q^n}$-partite graph whose parts are the members of a $(n-1)$-spread of PG$(\Fq^{mn})$
into $\K_k$-subspaces of PG$(\F_q^{mn})$.
These observations naturally lead to the following definitions.

\begin{defn}
Let $\Gamma$ be a simple graph. A $(mn,n,\Gamma,\lambda)$-GDD is a triple $({\cal P},{\cal G},{\cal B})$
where $\cal P$ is a set of $mn$ points, ${\cal G}$ is a partition of ${\cal P}$ into $m$ classes of size $n$,
and ${\cal B}$ is a collection of $\Gamma$-subgraphs ({\it blocks}) of $\K_{\cal P}$ such that the two vertices of any edge of any block belong to 
distinct classes, and two points belonging to different classes are adjacent in exactly $\lambda$ blocks.
\end{defn}

\begin{defn} Let $\Gamma$ be a $q$-spaceable graph.
A $(mn,n,\Gamma,\lambda)$-GDD over $\F_q$ is a $([mn]_q,[n]_q,\Gamma,\lambda)$-GDD
whose points are those of PG$(\F_q^{mn})$, whose classes are the members of a $(n-1)$-spread, and 
whose blocks are $\Gamma$-subspaces of PG$(\F_q^{mn})$.
\end{defn}

Here is a very elementary but useful composition construction.
\begin{prop}
If there exists both a $(mn,n,\Gamma,\lambda)$-GDD over $\F_q$ and 
a $2$$-(n,\Gamma,\lambda)$ design over $\F_q$, then there exists a $2$$-(mn,\Gamma,\lambda)$ design over $\F_q$.
\end{prop}
\begin{proof}
Let $({\cal P},{\cal S},{\cal B})$ be a $(mn,n,\Gamma,\lambda)$-GDD over $\F_q$. Each $S\in{\cal S}$ is a PG$(\F_q^n)$
and then, by assumption, we can construct a 2$-(n,\Gamma,\lambda)$ design over $\F_q$, say $(S,{\cal B}_S)$.
It is then clear that $({\cal P},{\cal B} \ \cup \ \bigcup_{S\in{\cal S}}{\cal B}_S)$ is the required $2$$-(mn,\Gamma,\lambda)$ design over $\F_q$.
\end{proof}

We already commented that every spanning 2$-([n]_q,\Gamma,\lambda)$ design can be seen as a 2$-(n,\Gamma,\lambda)$ design over $\F_q$. 
These designs, apparently uninteresting, could be 
crucial for the construction of some  2$-(mn,\Gamma,\lambda)$ designs over $\F_q$.
Indeed, as an immediate consequence of the previous proposition we can state the following.

\begin{cor}\label{notsotrivial}
If there exist a $(mn,n,\Gamma,\lambda)$-GDD over $\F_q$ and 
a spanning $2-([n]_q,\Gamma,\lambda)$ design, then there exists a $2$$-(mn,\Gamma,\lambda)$ design over $\F_q$.
\end{cor}

\section{Graph decompositions over finite fields by difference methods}
An automorphism of a $2$-design ${\cal D}=({\cal P},{\cal B})$, possibly over a finite field, is a bijection $\alpha: \cal P \longrightarrow {\cal P}$ 
preserving $\cal B$. Note that if ${\cal D}$ is over a finite field, then $\alpha$ necessarily maps subspaces into subspaces  and therefore
it necessarily belongs to the projective general linear group P$\Gamma$L$_v(q)$.
The set Aut$({\cal D})$ of all automorphisms of ${\cal D}$ is {\it the full automorphism group of ${\cal D}$} and it is clearly a subgroup of the symmetric group Sym$({\cal P})$.
If ${\cal D}$ is over a finite field, from what we have said above we have Aut$({\cal D})\leq$ P$\Gamma$L$_v(q)\leq $ Sym$({\cal P})$.
A 2$-(v,\Gamma,\lambda)$ design ${\cal D}=({\cal P},{\cal B})$ is {\it cyclic} if Aut$({\cal D})$ has a cyclic subgroup acting sharply transitively on $\cal P$. 
 
Using techniques based on automorphism groups of objects is not a novelty in the construction of combinatorial structures. 
We highlight a few of them used to construct designs over finite fields: the Kramer-Mesner method \cite{BKL05, BKW18-2}; the tactical-decomposition method \cite{NP15, BKW18-2};
the method of differences \cite{BN}. The last method will be used here to obtain some non-trivial cyclic graph decompositions over a finite field.

\subsection{Difference families}

\begin{defn}\label{DefDF}
Let $G$ be a group of order $v$ and let $\Gamma$ be a simple graph.  A $(v,\Gamma,\lambda)$ {\it difference family} in $G$ is a collection $\cal F$ 
of $\Gamma$-subgraphs of $\K_G$ ({\it base blocks}) such that $\Delta{\cal F}$ covers exactly $\lambda$ times the set $G^*$ of non-identity elements of $G$.
\end{defn}

If $\Gamma$ has size $s$, then the list of differences of a $\Gamma$-subgraph of $\K_G$ has size $2s$ and then
it is evident that a necessary condition for the existence of a $(v,\Gamma,\lambda)$ difference family is that $\lambda(v-1)$
is divisible by $2s$.
In the case that $\Gamma$ is the complete graph $\K_k$, one simply speaks of a $(v,k,\lambda)$
difference family in $G$. If we speak of a $(v,\Gamma,\lambda)$ difference family or a $([v]_q,\Gamma,\lambda)$
difference family without specifying the group $G$, it will be understood that $G=\Z_v$ or $G=[\Z_v]_q$, respectively.

The notion of a difference family is important in view of the following result
that is very well known when $\Gamma$ is complete (see, e.g., \cite{AB,BJL}). For a generic $\Gamma$ one can see \cite{BP,Weil}.

\begin{thm}\label{df}
If $\cal F$ is a $(v,\Gamma,\lambda)$ difference family in $G$, then the pair $(G,dev{\cal F})$
is a cyclic $2-$$(v,\Gamma,\lambda)$ design.
\end{thm}

The following definition is the $q$-analog of Definition \ref{DefDF}.

\begin{defn}
Let $q$ be a prime power and let $\Gamma$ be a $q$-spaceable graph.
A $(v,\Gamma,\lambda)$ difference family {\it over $\F_q$} is a $([v]_q,\Gamma,\lambda)$ difference family
in which every base block is a $\Gamma$-subspace of PG$(\F_q^v)$.
\end{defn}

Consistently with Notation \ref{q-troubles}, speaking of a $(v,C_k,\lambda)_q$ difference family
we will mean a $(v,C_{[k]_q},\lambda)$ difference family over $\F_q$. In particular, $(v,k,\lambda)_q$ difference family will mean
$(v,\K_k,\lambda)$ difference family over $\F_q$.

As a special case of Theorem \ref{df} we can state the following.

\begin{thm}\label{df_q}
The development of a $(v,\Gamma,\lambda)$ difference family over $\F_q$ is a $2$$-(v,\Gamma,\lambda)$ design over $\F_q$.
\end{thm}

The above theorem has been recently used in \cite{BN} to prove the existence of a cyclic 2$-(v,3,7)_2$ design -- that is  a 
2$-(v,\K_3,7)_2$ design -- for every odd $v$. That was an improvement of \cite{ST87} where the same result
was obtained with a different approach and the additional hypothesis that $v$ was not divisible by 3.
We also recall that all the 2$-(13,3,1)_2$ designs discovered in \cite{BEO16} are obtainable via 
$(13,3,1)_2$ difference families. We will revisit one of these difference families in Subsection \ref{revisiting}.

\subsection{Relative difference families}

Here we consider an important variation of a $(v,\Gamma,\lambda)$ difference family, that is the notion
of a {\it relative difference family}.

If $H$ is a subgroup of a group $G$, we denote by $\K_{G:H}$ the complete multipartite graph 
whose parts are the right cosets of $H$ in $G$.

\begin{defn}
Let $H$ be a subgroup of order $n$ of a group $G$ of order $mn$, and let $\Gamma$ be a simple graph.  
A $(mn,n,\Gamma,\lambda)$ {\it difference family in $G$ and relative to $H$} is a collection $\cal F$ of $\Gamma$-subgraphs 
of $\K_{G:H}$ such that $\Delta{\cal F}$ covers $G\setminus H$ exactly $\lambda$ times.
\end{defn}
Note that the list of differences of a difference family as above is clearly disjoint with $H$.
Thus, if $\Gamma$ has size $s$, then the obvious necessary condition for the existence of a $(mn,n,\Gamma,\lambda)$ 
difference family is that $\lambda(m-1)n$ is divisible by $2s$.
Of course, a  $(mn,1,\Gamma,\lambda)$ difference family relative to the trivial subgroup of $G$
is nothing but a  $(mn,\Gamma,\lambda)$ difference family in $G$ as defined in the previous section.

Speaking of a $(mn,n,\Gamma,\lambda)$ difference family or a $([mn]_q,[n]_q,\Gamma,\lambda)$
difference family without specifying the group $G$ and the subgroup $H$, it will be understood that $(G,H)=(\Z_{mn},m\Z_{mn})$ 
in the former case, and that $(G,H)=([\Z_{mn}]_q,[m\Z_{mn}]_q)$ in the latter.

The members of a relative difference family are called {\it base blocks} as for ordinary difference families.
Here we are interested in {\it relative difference families over finite fields}.

\begin{defn}
Let $q$ be a prime power and let $\Gamma$ be a $q$-spaceable graph.
A {\it $(mn,n,\Gamma,\lambda)$ difference family over $\F_q$} is a $([mn]_q,[n]_q,\Gamma,\lambda)$ difference family
whose base blocks are $\Gamma$-subspaces of PG$(\F_q^{mn})$.
\end{defn}

Consistently with Notation \ref{q-troubles}, speaking of a $(mn,n,C_k,\lambda)_q$ or a\break $(mn,n,P_k,\lambda)_q$ difference family, we will mean a
$(mn,n,\Gamma,\lambda)$ difference family over $\F_q$ where $\Gamma$ is the cycle or the path of order $[k]_q$,
respectively.

We have the following result.

\begin{thm}\label{rdf}
If $\cal F$ is a $(mn,n,\Gamma,\lambda)$ difference family, then $dev{\cal F}$
is a cyclic $(mn,n,\Gamma,\lambda)$-GDD.
\end{thm}

For the important case that $\Gamma$ is complete see \cite{recursive}, for a general $\Gamma$ see \cite{BG,BPbica,Weil}.
As a special case of the above theorem we can state the following.

\begin{thm}\label{rdf_q}
The development of a $(mn,n,\Gamma,\lambda)$ difference family over $\F_q$
is a $(mn,n,\Gamma,\lambda)$-GDD over $\F_q$.
\end{thm}

The above theorem has been used in \cite{BN} to prove the existence of a cyclic $(3n,3,3,7)_2$-GDD -- that is  a 
$(3n,3,\K_7,7)$ design over $\F_2$ -- for every odd $n$. 

As an immediate consequence of Theorem \ref{rdf_q} and Corollary \ref{notsotrivial} we can state the following.
\begin{prop}\label{spanning}
If there exists a $(mn,n,\Gamma,\lambda)$ difference family over $\F_q$ and 
a spanning $2-([n]_q,\Gamma,\lambda)$ design, then there exists a $2-(mn,\Gamma,\lambda)$ design over $\F_q$
which is cyclic if the spanning design has this property.
\end{prop}

\subsection{Use of multipliers}
Let $\cal F$ be a difference family in a group $G$ and let $\alpha$ be an automorphism of $G$.
One says that $\alpha$ is a {\it multiplier} of $\cal F$ if it leaves $\cal F$ invariant. 
Then the multipliers of $\cal F$ clearly form a subgroup of the automorphism group
of the design $\cal D$ generated by $\cal F$. Of course, it is not said that every automorphism of $\cal D$ is a composition
of a translation with a multiplier. 
Indeed the normalizer of the group of translations in the symmetric group on $G$ may contain elements that are not multipliers.

Using a difference family $\cal F$ in a group $G$ to construct a design $\cal D$
significantly reduces the number of blocks one needs to find.
Yet, $|{\cal F}|$ can be still quite ``big", hence the problem could appear to be hard anyway. 
So one could try to impose that $\cal F$ has a big group $A$ of multipliers
with a ``small" number of orbits (possibly one!) on $\cal F$. 
In this case it is enough to give a set of {\it initial base blocks} for $\cal F$, i.e., a complete system $\cal S$ 
of representatives for the $A$-orbits on the base blocks of $\cal F$; only one block, chosen arbitrarily, 
in each $A$-orbit on $\cal F$.

In Section \ref{Singersection} we will see how the construction of some ``difference graphs" (that are difference families with
only one base block) is facilitated if one imposes a group of multipliers.

Most constructions for difference families in a group $G$ have a group $A$ of multipliers 
acting {\it semiregularly} on $G^*$, i.e., on $G$ minus its the identity element. This means that the non identity elements of $A$ do not fix any element of $G^*$.
For instance, in \cite{disjoint} it is proved that there exists a {\it disjoint} $(v,k,k-1)$ difference family in $G$ whenever
Aut$(G)$ has a subgroup $A$ of order $k$ acting semiregularly on $G^*$. The base blocks of this difference family $\cal F$
are simply the $A$-orbits on $G^*$, hence $A$ is a group of multipliers of $\cal F$ fixing every base block.

More frequently, the construction of an ordinary difference family $\cal F$ in a group $G$ can be realized by imposing a group $A$ of multipliers
acting semiregularly both on $G^*$ and $\cal F$. This strategy is often successful when $G$ is elementary abelian, i.e., the additive group of a finite field. 

As far as we are aware, a formal description of how this strategy works in the general case is lacking. 
We give this description in the proof of the following theorem.

\begin{thm}\label{multiplier1}
Let $\Gamma$ be a graph of size $s$ and let $G$ be a group of order $v$ with $\lambda(v-1)=2st$.
Assume that $A$ is a subgroup of $Aut(G)$ of order a divisor $d$ of $\gcd(v-1,t)$ acting semiregularly on $G^*$.
Also assume that ${\cal I}$ is a ${t\over d}$-collection of $\Gamma$-subgraphs of $\K_G$ with $\Delta{\cal I}$ evenly distributed over the 
${v-1\over d}$ orbits of $A$ on $G^*$. Then ${\cal I}$ is a collection of initial base blocks of a $(v,\Gamma,\lambda)$ difference family in $G$.
\end{thm}
\begin{proof}
Set ${\cal F}=\{B^\alpha \ | \ B\in{\cal I},\alpha\in A\}$. We have $\Delta B^\alpha=\{x^\alpha \ | \ x\in \Delta B\}$ for every pair $(B,\alpha)\in{\cal I}\times A$ and then
$\Delta{\cal F}=\{x^\alpha \ | \ x\in \Delta{\cal I},\alpha\in A\}$. We have $|\Delta{\cal I}|=2s|{\cal I}|={2st\over d}=\lambda{v-1\over d}$ and then
$\Delta{\cal I}$ has exactly $\lambda$ elements in each $A$-orbit on $G^*$ by our assumption that $\Delta{\cal I}$ is evenly distributed over the orbits of $A$.
This means that $\Delta{\cal I}$ is the multiset sum of $\lambda$ complete systems of representatives for the $A$-orbits on $G^*$, 
say ${\cal S}_1$, \dots, ${\cal S}_\lambda$.  
Thus we can write $$\Delta{\cal F}=\biguplus_{i=1}^\lambda\biguplus_{x\in S_i}\{x^\alpha \ | \ \alpha\in A\}.$$
Let Stab$(x)$ and Orb$(x)$ be the stabilizer and the orbit of $x$ under the action of $A$. Then $\{x^\alpha \ | \ \alpha\in A\}$ is the multiset sum of 
${|A|\over |Stab(x)|}$ copies of Orb$(x)$. On the other hand $A$ acts semiregularly on $G^*$ by assumption, hence Stab$(x)$ is always
trivial and then $\{x^\alpha \ | \ \alpha\in A\}=Orb(x)$ for every $x\in S_i$. Thus we have
$\biguplus_{x\in S_i}\{x^\alpha \ | \ \alpha\in A\}=G^*$ for $1\leq i\leq \lambda$ and
we conclude that $\Delta{\cal F}$ is the multiset sum of $\lambda$ copies of $G^*$, i.e., $\cal F$ is a $(v,\Gamma,\lambda)$ difference family in $G$.
The assertion follows.
\end{proof}

Assume, for instance, that $q=k(k-1)t+1$ is a prime power and that we want to find a $(q,k,1)$ difference family in the elementary 
abelian group $EA(q)$, that is the additive group of $\F_q$. Thus we want a
$(q,\Gamma,1)$ difference family in $EA(q)$ where $\Gamma=\K_k$ has size $s={k(k-1)\over2}$. 
Let $C$ be the subgroup of $t$-th roots of unity of $\F_q^*$ and set
$A=\{\alpha_c \ | \ c\in C\}$ where $\alpha_c$ is the automorphism of EA$(q)$ defined by $\alpha_c(x)=cx$ for every $x\in\F_q$.
It is obvious that $A$ is a group of automorphisms of EA$(q)$ isomorphic to $C$ that acts semiregularly on $\F_q^*$
and that the $A$-orbits on $\F_q^*$ are the $k(k-1)$ cosets of $C$ in $\F_q^*$.
Thus, if we find a $k$-subset $B$ of $\F_q$ such that $\Delta B$ has exactly one element
in each of these cosets, then a set $\cal I$ of initial base blocks for the required family is the singleton $\{B\}$ by Theorem \ref{multiplier1}.
This is the famous ``Wilson's lemma on evenly distributed differences" \cite{W72}. At first sight one could think that to
find such a set $B$ is almost a miracle but, as proved by Wilson himself, this strategy always succeeds 
whenever $v$ is sufficiently large (see also \cite{Weil}).

It is easy to see that Theorem \ref{multiplier1} can be generalized to the following.

\begin{thm}\label{multiplier2}
Let $\Gamma$ be a graph of size $s$, let $G$ be a group of order $mn$ with $\lambda(m-1)n=2st$, and let $H$ be a subgroup of $G$ of order $n$.
Assume that $A$ is a subgroup of $Aut(G)$ of order a divisor $d$ of $\gcd(mn-n,t)$ acting semiregularly on $G\setminus H$.
Also assume that ${\cal I}$ is a ${t\over d}$-collection of $\Gamma$-subgraphs of $\K_{G:H}$ with $\Delta{\cal I}$ evenly distributed over the 
${(m-1)n\over d}$ orbits of $A$ on $G\setminus H$. Then ${\cal I}$ is a collection of initial base blocks of a $(mn,n,\Gamma,\lambda)$ difference family in $G$
relative to $H$.
\end{thm}

The above two theorems can be reformulated -- mutatis mutandis -- almost in the same way for difference families over a finite field but now there is a very a big ``handicap";
indeed in this case the subgroup $A$ of Aut$([\Z_v]_q)$ cannot be arbitrary since it must map subspaces into subspaces. This may happen only if
$A$ is a subgroup of the (unfortunately quite ``small") group $$Frob([\Z_v]_q)=\{\phi^i \ | \ 0\leq i\leq v-1\}$$ 
where $\phi$ is the {\it Frobenius automorphism} defined by
$\phi(x)=x^q$ for every $x\in [\Z_v]_q$.

A further inconvenience is that $Frob([\Z_v]_q)$ may not have any subgroup acting semiregularly on the complement
of the subgroup $H$ of $[\Z_v]_q$ that one needs.
Let us examine, for instance, what happens in the case that $H=\{1\}$.

\begin{prop}
$Frob([\Z_v]_q)$ has a non-trivial subgroup acting semiregularly on $[\Z_v]_q\setminus\{1\}$
if and only if $v$ is a prime and $q\not\equiv1$ $($mod $v)$.
\end{prop}
\begin{proof}
Let $F$ be the group of units of $\Z_{[v]_q}$ that is the image of $Frob([\Z_v]_q)$ under the
the natural isomorphism between $[\Z_v]_q$ and $\Z_{[v]_q}$, hence $F=\{q^i \ | \  0\leq i\leq v-1\}$.
We have to show, equivalently, that $F$ has a subgroup $A$ acting semiregularly
on $\Z_{[v]_q}\setminus\{0\}$ if and only if $v$ is a prime and $q\not\equiv1$ $($mod $v)$.

$(\Longrightarrow).$\quad
Given that $F$ is the cyclic group of order $v$ generated by $q$, we have $A=\langle q^d\rangle$ 
for some divisor $d$ of $v$. We have $q^d\cdot{q^v-1\over q^d-1}=(q^v-1)+{q^v-1\over q^d-1}\equiv {q^v-1\over q^d-1}$ (mod $[v]_q$),
hence ${q^v-1\over q^d-1}$ is fixed by $A$. It necessarily follows that ${q^v-1\over q^d-1}\equiv0$ (mod $[v]_q$)
and this is possible only for $d=1$. Thus $A=\langle q\rangle$, i.e., $A$ is necessarily the whole $F$.
This fact naturally implies that $F$ has prime order. Indeed, in the opposite case, any proper subgroup of $F$
would also act semiregularly on $\Z_{[v]_q}\setminus\{0\}$.

The $F$-orbits on $\Z_{[v]_q}\setminus\{0\}$ have all size $v$, hence we have $[v]_q-1\equiv0$ (mod $v$).
This gives $\sum_{i=1}^{v-1}q^i\equiv0$ (mod $v$) and then $q\not\equiv1$ (mod $v$)
otherwise we would have $v-1\equiv0$ (mod $v$) which is absurd.

$(\Longleftarrow).$\quad
Let $d=\gcd(q-1,[v]_q)$ so that we have $q\equiv1$ (mod $d$) and $[v]_q\equiv0$ (mod $d$).
Hence we can write $[v]_q=\sum_{i=0}^{v-1}q^i\equiv v$ (mod $d$). 
We deduce that $v\equiv0$ (mod $d$). Thus, given that $v$ is a prime, we have either $d=v$ or $d=1$.
The former case cannot happen since $q\not\equiv1$ $($mod $v)$ by assumption. We conclude that
$\gcd(q-1,[v]_q)=1$. Now assume that there is an element $x\in\Z_{[v]_q}$
whose $F$-stabilizer Stab$(x)$ is not trivial. Then, considering that $F$ has prime order, we necessarily have Stab$(x)=F$. 
This implies, in particular, that $qx\equiv x$ (mod $[v]_q$) so that
$(q-1)x$ is divisible by $[v]_q$. Thus, recalling that $[v]_q$ and $q-1$ are coprime, $[v]_q$ is a divisor of $x$, that means that $x$
is the zero element of $\Z_{[v]_q}$. We conclude that $F$ acts semiregularly on all elements of $\Z_{[v]_q}\setminus\{0\}$.
\end{proof}

In view of the above proposition, the ``$q$-analog" of Theorem \ref{multiplier1} should be more conveniently stated as follows.

\begin{thm}\label{multiplier3}
Let $v$ be a prime and let $\Gamma$ be a $q$-spaceable graph of size $s$ with $\lambda([v]_q-1)=2st$.
Assume that $v$ divides $t$ and that ${\cal I}$ is a ${t\over v}$-collection of $\Gamma$-subspaces of PG$(\F_q^v)$ with $\Delta{\cal I}$ evenly distributed over the 
${[v]_q-1\over v}$ orbits of Frob$([\Z_v]_q)$ on $[\Z_v]_q^*$. Then ${\cal I}$ is a collection of initial base blocks of a $(v,\Gamma,\lambda)$ difference family
over $\F_q$.
\end{thm}

The shortage of multipliers is one of the main reasons why constructing difference families over a finite field is in general a very hard task.
The search for difference families in a group $G$ could be enormously facilitated by the use of 
the automorphisms of $G$ to the point that, in some cases, it is enough to find only an initial base block for them. 
On the other hand, for a difference family $\cal F$ over a finite field the number of automorphisms that one can use is
very small compared with the size of $\cal F$ and hence, in general, one needs a huge number of initial base blocks anyway.

\subsection{An example of a $(7,Q_3^*,1)$ difference family over $\F_2$}\label{Q_3^*}
Let us show a concrete example where Theorem \ref{multiplier3} can be applied.
Other examples will be given in the next sections.

\medskip\noindent
Let $Q_3^*$ be the cube $Q_3$ with one vertex deleted
and let us construct a cyclic 2$-(7, Q_3^*,1)$ design over $\F_2$.
By Theorem \ref{df_q}, it is enough to exhibit a $(7,Q_3^*,1)$ difference family over $\F_2$.
Note that $Q_3^*$ has size $s=9$ and that we have $[7]_2-1=126=2st$ with $t=7$.
Thus, by Theorem \ref{multiplier3} we need only one initial $Q_3^*$-plane $B$ of PG$(\F_2^7)$
with the property that its list of differences has exactly one element in each orbit of Frob$([\Z_7]_2)$ on $[\Z_7]_2^*$.

Let us take a root $g$ of the polynomial $x^7+x+1$ as generator of $[\Z_7]_2$ and consider
the natural isomorphism $f$ between $[\Z_7]_2$ and $\Z_{127}$  mapping $g$ into 1.
Note that $p=127$ is a prime and that the image under $f$ of the orbits of Frob$([\Z_7]_2)$ on $[\Z_7]_2^*$
are the cosets of the group $\{2^i \ | \ 0\leq i\leq6\}$ of the 7th roots of unity in $\Z_p^*$, i.e., the {\it cyclotomic classes} of order 18.
Also note that $f$ maps the list of differences of $B$ into the list of differences of $f(B)$.
Thus a $Q_3^*$-plane $B$ satisfies our requirement provided that the list of differences of $B':=f(B)$ 
has exactly one element in each cyclotomic class of order 18. Let $r$ denote a primitive element of $\Z_p$. It is a standard exercise to verify that this
means that the logarithmic map 
$$Log: r^i \in \Z_{p}^* \longrightarrow i\in\Z_{18}$$
is bijective on $\Delta B'$.
Then our strategy to find the $Q_3^*$-plane $B$ is the following:
\begin{itemize}
\item[(i)] find a plane $\pi$ of PG$(\F_2^7)$ such that the list of differences of $f(\pi)$
intersects each cyclotomic class of order 18 in \underline{at least} one element or, equivalently,
in such a way that the map $Log$ is surjective on $\Delta f(\pi)$.
\item[(ii)] construct a copy $B'$ of $Q_3^*$ with vertex set $f(\pi)$ in such a way that 
the map $Log$ is injective on $\Delta B'$.
\end{itemize}
Consider the plane $\pi=\langle 1, g^2, g^5\rangle$ generated by the three points $1=g^0$, $g^2$ and $g^5$.
Taking into account of the algebraic rule $g^7=g+1$, one can easily check that for the remaining points of $\pi$, that are
$g^2+1$, $g^5+1$, $g^5+g^2$ and $g^5+g^2+1$, we have:
$$g^2+1=g^{14},\quad g^5+1=g^{54},\quad g^5+g^2=g^{65},\quad g^5+g^2+1=g^{95}.$$
Thus $f(\pi)=\{0,2,5,14,54,65,95\}$. 
The {\it difference table} of $f(\pi)$ is the following:

\medskip\small
\begin{center}
\renewcommand\arraystretch{1.0} 
\begin{tabular}{|c|c|c|c|c|c|c|c|c|}
\hline $$ & ${0}$ & ${2}$ & ${5}$ & ${14}$ & ${54}$ & ${65}$ & ${95}$\\
\hline $0$ & $$ & ${\bf125}$  & ${\bf122}$ & $\bf113$ & ${\bf73}$ & ${\bf62}$ & ${\bf32}$ \\
\hline $2$ & $\bf2$ & $$  & $\bf124$ & $\bf115$ & $\bf75$ & $\bf64$ & $\bf34$ \\
\hline $5$ & $\bf5$ & $\bf3$ & $ $ & $\bf118$ & $\bf78$ & $\bf67$ & $\bf37$ \\
\hline $14$ & $\bf14$ & $\bf12$ & $\bf9$ & $ $ & $\bf87$ & $\bf76$ & $\bf46$  \\
\hline $54$ & $\bf54$ & $\bf52$ & $\bf49$ & $\bf40$ & $$ & $\bf116$ & $\bf86$  \\
\hline $65$ & $\bf65$ & $\bf63$ & $\bf60$ & $\bf51$ & $\bf11$ & $$ & $\bf97$  \\
\hline $95$ & $\bf95$ & $\bf93$ & $\bf90$ & $\bf81$ & $\bf41$ & $\bf30$ & $$  \\
\hline
\end{tabular} 
\end{center}
\normalsize
Choosing $r=3$ as a primitive element of $\Z_p$, one can see that
the image of the above table under the map $Log$ is
\begin{center}
\renewcommand\arraystretch{1.0} 
\begin{tabular}{|c|c|c|c|c|c|c|c|c|}
\hline $$ & ${0}$ & ${2}$ & ${5}$ & ${14}$ & ${54}$ & ${65}$ & ${95}$\\
\hline $0$ & $$ & ${\bf9}$  & ${\bf6}$ & $\bf16$ & ${\bf12}$ & ${\bf10}$ & ${\bf0}$ \\
\hline $2$ & $\bf0$ & $$  & $\bf10$ & $\bf10$ & $\bf13$ & $\bf0$ & $\bf2$ \\
\hline $5$ & $\bf15$ & $\bf1$ & $ $ & $\bf11$ & $\bf5$ & $\bf7$ & $\bf8$ \\
\hline $14$ & $\bf7$ & $\bf1$ & $\bf2$ & $ $ & $\bf6$ & $\bf12$ & $\bf13$  \\
\hline $54$ & $\bf3$ & $\bf4$ & $\bf14$ & $\bf15$ & $$ & $\bf5$ & $\bf17$  \\
\hline $65$ & $\bf1$ & $\bf9$ & $\bf16$ & $\bf3$ & $\bf14$ & $$ & $\bf7$  \\
\hline $95$ & $\bf9$ & $\bf11$ & $\bf17$ & $\bf4$ & $\bf8$ & $\bf16$ & $$  \\
\hline
\end{tabular} 
\end{center}
\normalsize
and then that condition (i) is satisfied; indeed each element of $\Z_{18}$ appears at least once in the
entries of the above table. Now we have to label the vertices of the abstract graph $Q_3^*$ with the points of $f(\pi)$ 
in order to get a graph $B'$ satisfying (ii). We claim that such a graph $B'$ is for instance the following.

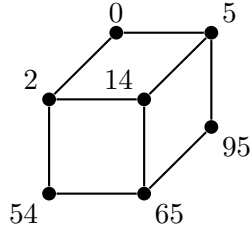
\begin{center}
	\begin{tikzpicture}[-,auto,node distance=2.5cm,thick,main node/.style={circle,fill=black,draw}]
	\node[circle,fill,scale=0.5](1) {};
	\node[circle,fill,scale=0.5, above right of=1](2) {};
	\node[circle,fill,scale=0.5, left of=1](3) {};
	\node[circle,fill,scale=0.5, above right of=3](4) {};
	\node[circle,fill,scale=0.5, below of=1](5) {};
	\node[circle,fill,scale=0.5, below of=3](6) {};
	\node[circle,fill,scale=0.5, below of=2](7) {};
	
	\path
	(1) edge node {} (2)
	(3) edge node {} (1)
	(4) edge node {} (2)
	(3) edge node {} (4)
	(3) edge node {} (6)
	(2) edge node {} (7)
	(1) edge node {} (5)
	(6) edge node {} (5)
	(5) edge node {} (7);
	
	\node [above left] at (1) {14};
	\node [above right] at (2) {5};
	\node [above left] at (3) {2};
	\node [above] at (4) {0};
	\node [below right] at (5) {65};
	\node [below left] at (6) {54};
	\node [below right] at (7) {95};
	\end{tikzpicture}
\end{center}

This is clearly recognizable looking at the image under $Log$ of the difference table of $B'$, that is the following.
\begin{center}
\renewcommand\arraystretch{1.0} 
\begin{tabular}{|c|c|c|c|c|c|c|c|c|}
\hline $$ & ${0}$ & ${2}$ & ${5}$ & ${14}$ & ${54}$ & ${65}$ & ${95}$\\
\hline $0$ & $$ & ${\bf9}$  & ${\bf6}$ & $$ & $$ & $$ & $$ \\
\hline $2$ & $\bf0$ & $$  & $$ & $\bf10$ & $\bf13$ & $$ & $$ \\
\hline $5$ & $\bf15$ & $$ & $ $ & $\bf11$ & $$ & $$ & $\bf8$ \\
\hline $14$ & $$ & $\bf1$ & $\bf2$ & $ $ & $$ & $\bf12$ & $$  \\
\hline $54$ & $$ & $\bf4$ & $$ & $$ & $$ & $\bf5$ & $$  \\
\hline $65$ & $$ & $$ & $$ & $\bf3$ & $\bf14$ & $$ & $\bf7$  \\
\hline $95$ & $$ & $$ & $\bf17$ & $$ & $$ & $\bf16$ & $$  \\
\hline
\end{tabular} 
\end{center}

\section{Steiner 2-designs over finite fields}

A 2$-(v,k,\lambda)$ design is said to be a {\it Steiner $2$-design} when $\lambda=1$.
In this section we will discuss Steiner 2-designs over $\F_q$, namely 2$-(v,k,1)_q$ designs or also 2$-(v,\K_k,1)_q$ designs.
It has already bee observed that a  2$-(v,3,1)_q$ design possibly exists only if $v\equiv1$ or 3 (mod 6) (see, e.g., \cite{BKN16}).
As far as we are aware, nobody noticed that, much more generally, the trivial admissibility conditions for the 
existence of a 2$-(v,k,1)_q$ design can be stated in the following very convenient and simple way. 

\begin{thm}\label{steineradmissibility1}
A $2$$-(v,k,1)_q$ design exists only if $v\equiv1$ or $k$ $($mod $k(k-1))$.
\end{thm}

This is as an immediate consequence of the following more general fact.

\begin{prop}\label{steineradmissibility2}
A classic $2-({q^v-1\over q-1},{q^k-1\over q-1},1)$ design exists only if $v\equiv1$ or $k$ $($mod $k(k-1))$.
\end{prop}
\begin{proof}
First recall that $\gcd(a^m-1,a^n-1)=a^{\gcd{(m,n)}}-1$ and therefore that
$\gcd({a^m-1\over a-1},{a^n-1\over a-1})={a^{\gcd(m,n)}-1\over a-1}$
for every triple of positive integers $(a,m,n)$ with $a>1$. 
This is a standard exercise of elementary number theory (see, e.g., \cite{Santos},
Example 245, page 36).
Specializing this to the case
that $a$ is a prime power, we can say that 
\begin{equation}\label{gcd}\gcd([m]_q,[n]_q)=[\gcd(m,n)]_q\end{equation} 
for every prime power $q$ and every pair $(m,n)$ of positive integers. This fact
implies, in particular, that 
\begin{equation}\label{[m]_q divides [n]_q}
[m]_q \ {\rm divides} \ [n]_q \Longleftrightarrow m  \ {\rm divides} \ n
\end{equation}

Indeed $[m]_q$ is a divisor of $[n]_q$ if and only if $\gcd([m]_q,[n]_q)=[m]_q$ which, by (\ref{gcd}), is equivalent to
say that $[\gcd(m,n)]_q=[m]_q$. Hence $[m]_q$ is a divisor of $[n]_q$ iff $\gcd(m,n)=m$, i.e., iff $m$ is a divisor of $n$.

Now assume that a 2$-([v]_q,[k]_q,1)$ design exists.
Here the divisibility conditions give ${[v]_q-1\over [k]_q-1}\in\N$ and ${[v]_q([v]_q-1)\over [k]_q([k]_q-1)}\in\N$ which, by trivial computation, mean
\begin{equation}\label{admissibleS_q}
{[v-1]_q\over [k-1]_q}\in \N \quad\quad {\rm and} \quad\quad {[v]_q[v-1]_q\over [k]_q[k-1]_q}\in \N
\end{equation}

By (\ref{[m]_q divides [n]_q}) and the first condition above, $k-1$ must be a divisor of $v-1$, i.e., $v\equiv1$ (mod $k-1$).

By (\ref{gcd}) we have $\gcd([v]_q,[v-1]_q)=[\gcd(v,v-1)]_q=[1]_q=1$, i.e., $[v]_q$ and $[v-1]_q$ are relatively prime.
Thus, given that $[k]_q$ is a divisor of their product by the second condition in (\ref{admissibleS_q}), $[k]_q$ is the
product of $\gcd([k]_q,[v]_q)$ and $\gcd([k]_q,[v-1]_q)$. It follows, by (\ref{gcd}), that $[k]_q=[g]_q\cdot[g']_q$ where
$g=\gcd(k,v)$ and $g'=\gcd(k,v-1)$.
Thus we can write:
\begin{equation}\label{brixen}
(q-1)(q^k-1)=(q^g-1)(q^{g'}-1).
\end{equation}
If $g\leq g'$, reducing (\ref{brixen}) modulo $q^g$ we get $q\equiv0$ (mod $q^g$) which implies $g=1$ and hence $g'=k$, i.e., $k$ divides $v-1$.
If $g\geq g'$, reducing (\ref{brixen}) modulo $q^{g'}$ we get $q\equiv0$ (mod $q^{g'}$) which implies $g'=1$ and hence $g=k$, i.e., $k$ divides $v$.
Thus we have $v\equiv0$ or 1 (mod $k$). 
Recalling that $v\equiv1$ (mod $k-1$), we necessarily conclude that $v\equiv1$ or $k$ $($mod $k(k-1))$
and the assertion follows.
\end{proof}

The following result is elementary and can be considered folklore.
Already in 1987, R. Mathon \cite{M} introduced it (end of page 353), without a proof, saying 
{\it ``An orbit analysis of a cyclic Steiner $C(v,k,1)$ yields the
following existence condition ..."}.
\begin{prop}\label{Mathon}
A cyclic $2$$-(v,k,1)$ design may exist only for $v\equiv 1$ or $k$ $($mod $k(k-1))$.
The block set of such a design is the development of a $(v,k,1)$ difference family when $v\equiv 1$ $($mod $k(k-1))$ or the development
of a $(v,k,k,1)$ difference family plus the cosets of ${v\over k}\Z_v$ in $\Z_v$ when $v\equiv k$ $($mod $k(k-1))$.
\end{prop}

By Theorem \ref{steineradmissibility1} and Proposition \ref{Mathon} the necessary conditions for the existence of a 2$-(v,k,1)_q$ design
exactly coincide with the necessary conditions for the existence of a cyclic classic 2$-(v,k,1)$ design. 
Thus, if $k$ is not a prime power, the ``2$-(v,k,1)_q$ admissibility conditions" 
are much more strict than the ``2$-(v,k,1)$ admissibility conditions". Consider, for instance, the case $k=6$.
The admissible values of $v$ for a classic 2$-(v,6,1)$ design are those congruent to 1, 6, 16, or 21 (mod 30).
On the other hand the admissible values of $v$ for a 2$-(v,6,1)_q$ design are only those congruent to 1 or 6 (mod 30).

Now note that $v\equiv 1$ or $k$ (mod $k(k-1)$) implies that $[v]_q\equiv 1$ or $[k]_q$ (mod $[k]_q([k]_q-1)$),
respectively.
Thus, by Theorem \ref{steineradmissibility1} and Proposition \ref{Mathon} again, 
it makes sense to try establishing the existence of a 2$-(v,k,1)_q$ design by searching for a cyclic example.  
Also, improving Theorem 7 in \cite{BEO16}, we can state the following.
\begin{thm}
There exists a cyclic $2$$-(v,k,1)_q$ design if and only if there exists a $(v,r,k,1)_q$ difference family where
$r$ is the remainder of the Euclidean division of $v$ by $k(k-1)$.
\end{thm}
A trivial counting shows that the size a difference family $\cal F$ as in the above theorem is
$$|{\cal F}|=\begin{cases}{(q-1)(q^{v-1}-1)\over(q^k-1)(q^{k-1}-1)}\quad\quad\mbox{if $v\equiv 1$ (mod $k(k-1)$)}\medskip\cr
{q^{k-1}(q-1)(q^{v-k}-1)\over(q^k-1)(q^{k-1}-1)} \quad\mbox{if $v\equiv k$ (mod $k(k-1)$)}\end{cases}$$
It is clear that this size is quite ``big" even for very ``small" values of the parameters $v$, $k$ and $q$.
Thus, it would be convenient to use multipliers, when this is possible.
Specializing Theorem \ref{multiplier3} to a 2$-(k(k-1)t+1,k,1)_q$ design (hence $\Gamma=\K_k$ and $\lambda=1$) we get the following.

\begin{thm}\label{multiplier4}
Let $v\equiv1$ $($mod $k(k-1))$ be a prime. Assume that ${[v-1]_q\over[k]_q[k-1]_q}=uv$ for some integer $u$ and that $q\not\equiv1$ $($mod $v)$.
If ${\cal I}$ is a $u$-set of $(k-1)$-dimensional subspaces of PG$(\F_q^v)$ with $\Delta{\cal I}$ evenly distributed over the 
${[v]_q-1\over v}$ orbits of Frob$([\Z_v]_q)$ on $[\Z_v]_q^*$, then ${\cal I}$ is a collection of initial base blocks of a $(v,k,1)_q$ difference family.
\end{thm}

\medskip
\subsection{Revisiting the $2$$-(13,3,1)_2$ design}\label{revisiting}

The longstanding conjecture that there is no non-trivial 2$-(v,k,1)_q$ design was disproved in  \cite{BEO16} 
where over 400 non-isomorphic cyclic 2$-(13,3,1)_2$ designs have been constructed. 
Given that they are cyclic, each of them can be obtained by means of a suitable $(13,3,1)_2$ difference family. 
Here we revisit the solution presented in \cite{BEO16} giving some more details. 
Our purpose is to make the reader able
to check its correctness almost by hand and, above all, we want to emphasize how multipliers are 
crucial for its achievement.

Let us take a root $g$ of the polynomial $x^{13} + x^{12} + x^{10} + x^{9} + 1$ as generator of $[\Z_{13}]_2$ and let us
consider the natural isomorphism $f: g^i\in [\Z_{13}]_2 \longrightarrow i\in \Z_{2^{13}-1}$.
Note that $p=[13]_2=2^{13}-1=8191$ is a prime so that it makes sense to speak of a primitive root (mod $p$). 
Such a primitive root is, for instance, $r=17$.

Let us use Theorem \ref{multiplier4} with $v=13$, $k=3$ and $q=2$. We have ${[v-1]_q\over[k]_q[k-1]_q}={2^{12}-1\over7\cdot3}=195=15v$.
Thus the required difference family could be realized by means of a 15-set ${\cal I}$ of  
initial planes of PG$(\F_2^{13})$ with the property that $\Delta{\cal I}$ has exactly one element in each orbit 
of Frob$(\Z_2^{13})$ on $[\Z_{13}]_2^*$. 
Note that the images of these orbits under 
the isomorphism $f$ are the cosets of $\{2^i \ | \ 0\leq i\leq 12\}$, that is the group of 13th 
roots of unity, in $\Z_p^*$. Equivalently, they are the cyclotomic classes of order ${p-1\over13}=630$.
Reasoning as in subsection \ref{Q_3^*} we conclude that ${\cal I}=\{B_1,\dots,B_{15}\}$ is the required set of initial planes
provided that the logarithmic map
$$Log: 17^i\in \Z_p^* \longrightarrow i\in\Z_{630}$$
is bijective on $S:=\bigcup_{i=1}^{15}\Delta f(B_i)$. That said, the required set of initial planes is the one consisting of the
preimages under $f$ of the following subsets of $\Z_p$.
\scriptsize
$$B'_1=\{0, 1, 1249, 7258, 8105, 5040, 7978\},\quad B'_2=\{0, 7, 1857, 6681, 7259, 7381, 7908\},$$
$$B'_3=\{0, 9, 1144, 7714, 1945, 8102, 6771\},\quad B'_4=\{0, 11, 209, 1941, 3565, 6579, 2926\},$$
$$B'_5=\{0, 12, 2181, 3696, 6673, 6965, 2519\},\quad B'_6=\{0, 13, 4821, 8110, 8052, 5178, 7823\},$$
$$B'_7=\{0, 17, 291, 5132, 1199, 8057, 6266\},\quad B'_8=\{0, 20, 1075, 3996, 7313, 4776, 3939\},$$
$$B'_9=\{0, 21, 2900, 6087, 4915, 4226, 8008\},\quad B'_{10}=\{0, 27, 1190, 3572, 6710, 4989, 5199\},$$
$$B'_{11}=\{0, 30, 141, 682, 6256, 6406, 2024\},\quad B'_{12}=\{0, 31, 814, 1243, 4434, 1161, 6254\},$$
$$B'_{13}=\{0, 37, 258, 5396, 6469, 2093, 4703\},\quad B'_{14}=\{0, 115, 949, 1272, 4539, 4873, 1580\},$$
$$B'_{15}=\{0, 119, 490, 6670, 6812, 7312, 5941\}.$$

\normalsize
The reader who wants to check this concretely, has to make the following three steps.

First of all one needs to ensure that the preimages of the $B'_i$s are actually planes. 
To facilitate this task each $B'_i$ has been ordered -- differently from \cite{BEO16} where the order is increasing --
in the form $B'_i=\{0,b_{i1},\dots,b_{i6}\}$ in such a way that 
$$g^{b_{i3}}=g^{b_{i1}}+1,\quad g^{b_{i4}}=g^{b_{i2}}+1,\quad g^{b_{i5}}=g^{b_{i1}}+g^{b_{i2}},\quad g^{b_{i6}}=g^{b_{i1}}+g^{b_{i2}}+1.$$
This can be easily verified by using the identity $g^{13}=g^{12} + g^{10} + g^{9} + 1$. Hence 
$B_i:=f^{-1}(B'_i)$ is the plane $\langle1, g^{b_{i1}}, g^{b_{i2}}\rangle$ generated by the three points $1$, $g^{b_{i1}}$ and $g^{b_{i2}}$.

Then one has to calculate $Log(\Delta B'_i)$ for each $i$. Here is, for instance,
the image under $Log$ of the difference table of $B'_1$. 

\scriptsize\begin{center}
\renewcommand\arraystretch{1.0} 
\begin{tabular}{|c|c|c|c|c|c|c|c|c|}
\hline $$ & ${0}$ & ${1}$ & ${1249}$ & ${7258}$ & ${8105}$ & ${5040}$ & ${7978}$\\
\hline $0$ & $$ & ${\bf315}$  & ${\bf376}$ & $\bf460$ & $\bf343$ & $\bf230$ & $\bf325$ \\
\hline $1$ & $\bf0$ & $$  & $\bf454$ & $\bf547$ & $\bf203$ & $\bf94$ & $\bf540$ \\
\hline $1249$ & $\bf61$ & $\bf139$ & $$ & $\bf265$ & $\bf288$ & $\bf328$ & $\bf344$ \\
\hline $7258$ & $\bf145$ & $\bf232$ & $\bf580$ & $ $ & $\bf478$ & $\bf100$ & $\bf105$  \\
\hline $8105$ & $\bf28$ & $\bf518$ & $\bf603$ & $\bf163$ & $$ & $\bf5$ & $\bf548$  \\
\hline $5040$ & $\bf545$ & $\bf409$ & $\bf13$ & $\bf415$ & $\bf320$ & $$ & $\bf308$  \\
\hline $7978$ & $\bf10$ & $\bf225$ & $\bf29$ & $\bf420$ & $\bf233$ & $\bf623$ & $$  \\
\hline
\end{tabular} 
\end{center}

\normalsize
Finally, one has to check that the ``miracle" happens:
the union of these images cover all $\Z_{630}$.

\medskip
The success of the above construction actually looks like a miracle. It is even more amazing 
that, in the same way, more than 400 pairwise non-isomorphic 2$-$$(13,3,1)_2$ designs have been obtained.
This fact seems to suggest that there is a ``magic" combinatorial structure on $2^{13}-1=8191$ points hidden behind these designs. Thus, given that $8191=90^2+90+1$, we hazard the outrageous conjecture that there exists a projective plane of order 90.

\subsection{Searching for other cyclic $2$$-(v,k,1)_q$ designs}

The existence problem for Steiner 2-designs over $\F_q$ is very hard.
For the time being, the only theoretical tool available to get them is the 
method of differences that requires the construction of a (relative)
difference family whose size is almost always quite big. The possible existence of multipliers 
does not help so much since the number of initial base blocks usually remains too big.
Let us show, for instance, what happens if we want to find the {\it $q$-analog of a Fano plane}
that is a 2$-(7,3,1)_q$ design. In the following table we give the size of a putative 2$-(7,3,1)_q$ difference family $\cal F$
and the minimal size of a set $\cal I$ of initial base blocks for $\cal F$ for each prime power $q\leq19$.

\small\medskip\begin{center}
\renewcommand\arraystretch{1.3} 
\begin{tabular}{|c|c|c|c|c|c|c|c|c|c|c|c|}
\hline $q$ & $2$ & $3$ & $4$ & $5$ & $7$ & $9$ & $11$ & $13$ & $16$ & $17$ & 19\\
\hline $|{\cal F}|\atop{|{\cal I}|}$ & $3$ & $7\atop1$ & $13$ & $21\atop3$ & 43 & 73 & 111 & 157 & 241 & $273\atop39$ & $343\atop49$\\
\hline $\exists$ & N & N & ? & N &  ? & ? & ? & ? &  ? & $?\atop N$ & ? \\
\hline
\end{tabular}
\end{center}

\normalsize
If the cell $({|{\cal F}|\atop{|{\cal I}|}},q)$ contains two numbers $f\atop i$ it means that $f$ is the size of ${\cal F}$ and that
$i$ is the number of initial base blocks for $\cal F$ in the putative case that it admits a group of multipliers.
If the cell $({|{\cal F}|\atop{|{\cal I}|}},q)$ contains only one number $f$ it means that ${\cal F}$ has size $f$
and cannot have a non-trivial group of multipliers.

In the line labeled ``$\exists$" we put ``N" or ``$?$" according to whether the non-existence has been checked by computer
(see \cite{BEO16} for the cases $q\in\{2,3,5\}$)
or it is still undecided, respectively. Note that we have checked by computer that a putative 2$-(7,3,1)_{17}$ difference family 
does not have non-trivial multipliers even though, a priori, it may have Frob$[\Z_7]_{17}$ as a group
of multipliers.

For $(v,k)\neq(7,3)$ the situation becomes even worse.
We first recall that an exhaustive computer search \cite{BEO16} excluded the existence of cyclic $(9,3,1)_2$ and $(13,4,1)_2$ designs.
In the following table we report the size of a putative $(v,r,k,1)_q$ difference family $\cal F$ and the size of a minimal set $\cal I$ of initial 
base blocks for $\cal F$. It is impressive how these numbers almost immediately ``explode". When the cell corresponding to ${\cal I}$ is empty it
will mean that $\cal F$ cannot have non-trivial multipliers.

\scriptsize\medskip\small\noindent
\renewcommand\arraystretch{1.0} 
\begin{tabular}{|c|c|c|c|c|c|c|c|c|}
\hline $v$ & $k$ & $q$ & $|{\cal F}|$ & $|{\cal I}|$\\
\hline $\bf13$ & $\bf3$ & $\bf2$ & $\bf195$ & $\bf15$\\
\hline $15$ & $3$ & $2$ & $780$ & $260$\\
\hline $19$ & $3$ & $2$ & $12483$ & $657$\\
\hline $21$ & $3$ & $2$ & $49932$ & $16644$\\
\hline\hline
\hline $13$ & $3$ & $3$ & $5110$ & $$\\
\hline $19$ & $3$ & $3$ & $3725197$ & $196063$\\
\hline $21$ & $3$ & $3$ & $33526773$ & $111755591$\\
\hline $25$ & $3$ & $3$ & $2715668620$ & $$\\
\hline
\end{tabular}\hfill
\begin{tabular}{|c|c|c|c|c|c|c|c|c|}
\hline $v$ & $k$ & $q$ & $|{\cal F}|$ & $|{\cal I}|$\\
\hline $16$ & $4$ & $2$ & $312$ & $$\\
\hline $25$ & $4$ & $2$ & $159783$ & $$\\
\hline $28$ & $4$ & $2$ & $1278264$ & $319566$\\
\hline $13$ & $4$ & $3$ & $511$ & $$\\
\hline\hline
\hline $16$ & $4$ & $3$ & $13797$ & $$\\
\hline $25$ & $4$ & $3$ & $271566862$ & $$\\
\hline $28$ & $4$ & $3$ & $7332305274$ & $$\\
\hline $37$ & $4$ & $3$ & $>10^{15}$ & $>10^{13}$\\
\hline
\end{tabular} 

\normalsize
\section{Cycle decompositions over finite fields}\label{cyclesection}

The admissibility conditions for the existence of a cycle decomposition over a finite field are the following.
\begin{prop}
A $2-$$(v,C_k,1)_q$-design with $q$ even possibly exists only for $v\equiv0$ or $1$ $($mod $k)$. 
A $2$$-(v,C_k,1)_q$-design with $q$ odd and $k$ even possibly exists only for $v\equiv1$ $($mod $k)$.
A $2$$-(v,C_k,1)_q$-design with both $q$ and $k$ odd possibly exists only for $v\equiv1$ or $k$ $($mod $2k)$.
\end{prop}
\begin{proof}
Here conditions (iii) and (iv) of Proposition \ref{admissible} are:
\begin{equation}
q[v]_q[v-1]_q\equiv0 \ (mod \ 2[k]_q) \quad{\rm and} \quad q[v-1]_q\equiv0  \ (mod \ 2).
\end{equation}
From the first congruence $[k]_q$ must be a divisor of $[v]_q[v-1]_q$ and then, reasoning as in the 
proof of Proposition \ref{brixen}, we get $v\equiv0$ or $1$ $($mod $k)$. 
If $q$ is odd, the second congruence gives $\displaystyle\sum_{i=0}^{v-2}q^i\equiv0$ (mod 2) 
and then $v-1\equiv0$ (mod 2), i.e., $v$ is odd. The assertion easily follows.
\end{proof}

As a consequence of the above proposition, one can try to construct every putative 2$-(v,C_k,1)_q$ design as follows.

\medskip\noindent
Case  $v\equiv1$ (mod $k$), say $v=kt+1$.

Find a collection ${\cal S}$ of ${1\over2}q[t]_{q^k}$ subspaces of PG$(\F_q^v)$ of dimension $k-1$ whose lists of differences 
cover, all together, every non-identity element of $[\Z_v]_q$ at least once.

Arrange the points of each $S\in{\cal S}$ into a $[k]_q$-cycle $C_S$ in such a way that 
${\cal F}:=\{C_S \ |  \ S\in S\}$ is a $(v,C_k,1)_q$ difference family. The development of $\cal F$ 
is the desired 2$-(v,C_k,1)_q$ design by Theorem \ref{df_q}.

\medskip\noindent
Case  $v\equiv0$ (mod $k$), say $v=kt$.

Find a collection ${\cal S}$ of ${1\over2}q^k[t-1]_{q^k}$ subspaces of PG$(\F_q^v)$ of dimension $k-1$ 
such that $\Delta{\cal S}$ covers every element of $[\Z_{kt}]_q\setminus [t\Z_{kt}]_q$ at least once. 

Arrange the points of each $S\in{\cal S}$ into a $[k]_q$-cycle $C_S$ in such a way that 
${\cal F}:=\{C_S \ |  \ S\in S\}$ is a $(v,k,C_k,1)_q$ difference family. At this point, let us recall that for every odd 
integer $u\geq3$ there exists a Hamiltonian cycle system of order $u$, i.e., a 2$-(u,C_u,1)$ design (see, e.g., \cite{endm}). 
Thus, in particular, there exists a spanning $([k]_q,C_{[k]_q},1)$ design over $\F_q$ and the existence of the desired 2$-(v,C_k,1)_q$ design
follows from Proposition \ref{spanning}.

Let us see how the above strategy is successful to find a $2$-$(v,C_3,1)_2$ design for $v=6, 7$ and $9$.

\medskip
{\bf A cyclic $2$$-(7,C_3,1)_2$ design}.

Let us take a root $g$ of the polynomial $x^7+x+1$ as generator of $[\Z_7]_2$. We need a $(7,C_3,1)_2$ difference family, 
namely a set ${\cal F}$ of nine $C_7$-planes of PG$(\F_2^7)$ whose list of differences covers $[\Z_7]_2\setminus\{1\}$ exactly once. 
We first need a set $\{\pi_1, \dots, \pi_9\}$ of nine planes of PG$(\F_2^7)$ forming a difference cover of $[\Z_7]_2\setminus\{1\}$.
We claim that such a difference cover is the one in which the $i$-th plane
$$\pi_i=\langle1,g^{x_i},g^{y_i}\rangle=\{1,g^{x_i},g^{y_i},g^{x_i}+1,g^{y_i}+1,g^{x_i}+g^{y_i},g^{x_i}+g^{y_i}+1\}$$
is generated by the three points $1$, $g^{x_i}$ and $g^{y_i}$ where
the pairs $(x_1,y_1)$, \dots, $(x_9,y_9)$ are as follows:
$$(1,3),\quad(1,71),\quad(2,18),$$
$$(2,22),\quad(2,41),\quad(3,13),$$
$$(3,20),\quad(8,19),\quad(10,40).$$
Using the identity $g^7=g+1$, the reader can check that the images $f(\pi_1)$, \dots, $f(\pi_9)$ 
of the nine planes in $\Z_{[7]_2}$ are the following:
$$\{0,1,3,7,63,15,31\},\quad\{0,1,71,7,79,92,74\},\quad\{0,2,18,14,53,114,42\},$$
$$\{0,2,22,14,47,91,70\},\quad\{0,2,41,14,75,102,80\},\quad\{0,3,13,63,55,111,96\},$$
$$\{0,3,20,63,89,46,37\},\quad\{0,8,19,56,29,95,65\},\quad\{0,10,40,108,51,72,85\}.$$
Now arrange the points of each $f(\pi_i)$ into a 7-cycle $B_i$ as follows:
$$(0, 1, 3, 7, 15, 31, 63),\quad (0, 7, 1, 71, 74, 79, 92),\quad (0, 18, 42, 14, 2, 114, 53),$$ 
$$ (0, 14, 47, 70, 91, 2, 22),\quad (0, 80, 2, 75, 41, 14, 102),\quad (0, 55, 3, 111, 63, 13, 96),$$
$$(0, 29, 19, 8, 95, 65, 56),\quad (0, 37, 20, 89, 63, 3, 46),\quad (0, 51, 10, 72, 108, 40, 85).$$
The lists of differences $\Delta B_1$, \dots, $\Delta B_9$ of the above cycles are the following:
\small
$$\pm\{1,2,4,8,16,32,63\}\quad\pm\{7,6,57,3,5,13,35\},\quad\pm\{18,24,28,12,15,61,53\}$$
$$\pm\{14,33,23,21,38,20,22\}, \ \pm\{47,49,54,34,27,39,25\}, \ \pm\{55,52,19,48,50,44,31\},$$
$$\pm\{29,10,11,40,30,9,56\}, \ \pm\{37,17,58,26,60,43,46\}, \ \pm\{51,41,62,36,59,45,42\}.$$
\normalsize
We see that the above lists partition $\Z_{127}\setminus\{0\}$, hence $\{B_1,\dots,B_9\}$
is a $(127,C_7,1)$ difference family and then ${\cal F}=\{f^{-1}(B_1),...,f^{-1}(B_9)\}$ is a
$(7,C_3,1)_2$ difference family.

\medskip
{\bf A cyclic $2$$-(6,C_3,1)_2$ design}. 

Let us take a root $g$ of the polynomial $x^6+x^4+x^3+x+1$ as generator of $[\Z_6]_2$ and let $f$
be the natural isomorphism between $[\Z_6]_2$ and $\Z_{[6]_2}$ mapping $g$ to 1. Here we need a $(6,3,C_3,1)_2$ difference family, 
namely a set ${\cal F}$ of four $C_7$-planes of PG$(\F_2^7)$ whose list of differences covers $[\Z_6]_2\setminus[2\Z_6]_2$ exactly once. 
We first need a set $\{\pi_1, \dots, \pi_4\}$ of four planes of PG$(\F_2^6)$ forming a difference cover of $[\Z_6]_2$.
Such a difference cover is the one for which the $i$-th plane
$$\pi_i=\langle1,g^{x_i},g^{y_i}\rangle=\{1,g^{x_i},g^{y_i},g^{x_i}+1,g^{y_i}+1,g^{x_i}+g^{y_i},g^{x_i}+g^{y_i}+1\}$$
is generated by the three points $1$, $g^{x_i}$ and $g^{y_i}$ where
the pairs $(x_1,y_1)$, \dots, $(x_4,y_4)$ are as follows:
$$(1,21),\quad(2,21),\quad(2,9),\quad(9,12).$$
Using the identity $g^6+g^4+g^3+g+1=0$, the reader can check that the images $f(\pi_1)$, \dots, $f(\pi_4)$ 
of the four planes in $\Z_{[6]_2}$ are the following:
$$\{0,1,21,56,42,58,25\},\quad\{0,2,21,49,42,50,53\},$$
$$\{0,2,9,49,27,10,60\},\quad\{0,9,12,27,52,22,46\}.$$
Now arrange the points of each $f(\pi_i)$ into a 7-cycle $B_i$ as follows:
$$(0, 1, 58, 25, 21, 56, 42),\quad (0, 50, 42, 49, 2, 21, 53),$$
$$(0, 2, 27, 10, 49, 9, 60),\quad (0, 22, 27, 12, 46, 9, 52).$$ 
The lists of differences $\Delta B_1$, \dots, $\Delta B_4$ of the above cycles are the following:
$$\pm\{1,6,30,4,28,14,21\},\quad\pm\{13,8,7,16,19,31,10\},$$
$$\pm\{2,25,17,24,23,12,3\},\quad\pm\{22,5,15,29,26,20,11\}.$$
The above lists partition $\Z_{63}\setminus9\Z_{63}$, hence $\{B_1,\dots,B_4\}$
is a $(63,7,C_7,1)$ difference family and then ${\cal F}=\{f^{-1}(B_1),...,f^{-1}(B_4)\}$ is a
$(6,3,C_3,1)_2$ difference family.

\medskip
{\bf A cyclic 2$-(9,C_3,1)_2$ design}. 

Let us take a root $g$ of the polynomial $x^9+x^4+1$ as generator of $[\Z_9]_2$ and let $f$
be the natural isomorphism between $[\Z_9]_2$ and $\Z_{[9]_2}$ mapping $g$ to 1. Here we need a $(9,3,C_3,1)_2$ difference family, 
namely a set ${\cal F}$ of thirty-six $C_7$-planes of PG$(\F_2^9)$ 
whose list of differences covers $[\Z_9]_2\setminus[3\Z_9]_2$ exactly once. 
Note that Frob$([\Z_9]_2)$ acts semiregularly on $[\Z_9]_2\setminus[3\Z_9]_2$.
Thus, by a suitable specialization of Theorem \ref{multiplier2}, the required difference family
can be realized by means of a set ${\cal I}$ of four $C_7$-planes whose 56 differences form a complete system
of representatives for the orbits of Frob$([\Z_9]_2)$ on $[\Z_9]_2\setminus[3\Z_9]_2$.
One can check that such a set $\cal I$ is the one formed by the preimages of the following 
7-cycles of $\Z_{511}$:
$$(0, 60, 1, 470, 130, 11, 504), \quad  (0, 134, 130, 14, 1, 333, 139),$$
$$(0, 24, 130, 1, 294, 338, 474),\quad (0, 27, 130, 277, 1, 185, 142).$$

\section{Path decompositions over finite fields}
The admissibility conditions for the existence of a path decomposition over a finite field are the following.
\begin{prop}\label{path}
A $2$$-(v,P_k,1)_q$ design with $q$ even cannot exist.\break
A $2$$-(v,P_k,1)_q$ design with $q$ odd and $k$ even possibly exists only for $v\equiv0$ or $1$ $($mod $k-1)$. 
A $2$$-(v,P_k,1)_q$ design with $q$ odd and $k$ odd possibly exists only for $v\equiv0$ or $1$ $($mod $2(k-1))$. 
\end{prop}
\begin{proof}
A path with $n$ vertices has size $n-1$, hence the size of $[P_k]_q$ is $[k]_q-1=q[k-1]_q$.
Thus, if a $2$-$(v,P_k,1)_q$ design exists, condition (iii) of Proposition \ref{admissible} gives $q[v]_q[v-1]_q\equiv0$ (mod $2q[k-1]_q$), hence
$2[k-1]_q$ must be a divisor of $[v]_q[v-1]_q$.
It is obvious that this is not possible for $q$ even since in this case both $[v]_q$ and $[v-1]_q$ are odd.
Thus $q$ must be odd and $[k-1]_q$ must be a divisor of $[v]_q[v-1]_q$. Reasoning as in the 
proof of Proposition \ref{brixen}, we get $v\equiv0$ or $1$ $($mod $k-1)$. Thus we have $v=(k-1)t+r$ with $r=0$ or 1 for a suitable $t$
and a trivial counting shows that we have: $${[v]_q[v-1]_q\over 2[k-1]_q}={1\over2}\cdot\sum_{i=0}^{t-1}q^{(k-1)i}\cdot\sum_{i=0}^{(k-1)t-2r}q^i$$
The reduction (mod 2) of the two sums in the above formula are respectively equal to $t$ and $(k-1)t+1$. Thus, for $k$ odd, their product is
even only for $t$ even. The assertion follows.
\end{proof}

Differently from Steiner 2-designs and cycle decompositions over finite fields, we note that there are admissible triples $(v,k,q)$ for which, 
a priori, a 2$-(v,P_k,1)_q$ design cannot be obtained via difference families. The first of these triples is $(4,3,3)$; according to
Proposition \ref{path} a 2$-(4,P_3,1)_3$ design may exist but it does not make sense to speak of a $(4,P_3,1)_3$ difference family.

The ``smallest" admissible non-trivial triple $(v,k,q)$ for which a $(v,P_k,1)_q$ difference family may exist is $(5,3,3)$.
Thus let us construct a $(5,P_3,1)_3$ difference family, i.e., a set $\cal F$ of $P_{13}$-planes of PG$(3^5)$ 
whose list of differences covers $[\Z_5]_3^*$ exactly once. The size of $P_{13}$ is $s=12$ and 
we have $[5]_3-1=120=2st$ with $t=5$. Thus, by Theorem \ref{multiplier1}, a possible way to realize the difference family $\cal F$
is to look for only one $P_{13}$-plane of PG$(3^5)$ whose list of differences is a complete
system of representatives for the orbits of Frob$([\Z_5]_3)$ on $[\Z_5]_3^*$.

Let us take a root $g$ of the polynomial $x^5+2x+1$ as generator of $[\Z_5]_3$ and let $f$
be the natural isomorphism between $[\Z_5]_3$ and $\Z_{[5]_3}$ mapping $g$ to 1. 
Consider the plane $\pi=\langle 1, g, g^3\rangle$ generated by the three points $1=g^0$, $g$ and $g^3$. The remaining points of $\pi$ 
are the following:
$$g+1=g^{69}, \quad g+2=g^5, \quad g^3+1=g^{86}, \quad g^3+2=g^{15},$$ 
$$g^3+g=g^{47}, \quad g^3+g+1=g^{93}, \quad g^3+g+2=g^{49},$$ 
$$g^3+2g=g^{75}, \quad g^3+2g+1=g^{77}, \quad g^3+2g+2=g^{28}.$$

One can check that the list of differences of $\pi$ has at least one element in each of the 24 orbits of Frob$([\Z_5]_3)$ on $[\Z_5]_3^*$.
Then it makes sense to look for a $P_{13}$-plane $B$ with vertex set $\pi$ whose list of differences has exactly one element
in each of those orbits. Such a $B$ will be the desired $P_{13}$-plane.
The reader can easily recognize that a good $B$ is, for instance, the preimage under $f$ of the path depicted below. 
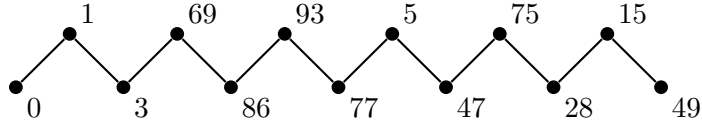
\begin{center}
	\begin{tikzpicture}[-,auto,node distance=2cm,thick,main node/.style={circle,fill=black,draw}]
	\node[circle,fill,scale=0.5](1) {};
	\node[circle,fill,scale=0.5, above right of=1](2) {};
	\node[circle,fill,scale=0.5, below right of=2](3) {};
	\node[circle,fill,scale=0.5, above right of=3](4) {};
	\node[circle,fill,scale=0.5, below right of=4](5) {};
	\node[circle,fill,scale=0.5, above right of=5](6) {};
	\node[circle,fill,scale=0.5, below right of=6](7) {};
	\node[circle,fill,scale=0.5, above right of=7](8) {};
	\node[circle,fill,scale=0.5, below right of=8](9) {};
	\node[circle,fill,scale=0.5, above right of=9](10) {};
	\node[circle,fill,scale=0.5, below right of=10](11) {};
	\node[circle,fill,scale=0.5, above right of=11](12) {};
	\node[circle,fill,scale=0.5, below right of=12](13) {};
	
	\path
	(1) edge node {} (2)
	(2) edge node {} (3)
	(3) edge node {} (4)
	(4) edge node {} (5)
	(5) edge node {} (6)
	(6) edge node {} (7)
	(7) edge node {} (8)
	(8) edge node {} (9)
	(9) edge node {} (10)
	(10) edge node {} (11)
	(11) edge node {} (12)
	(12) edge node {} (13);
	
	\node [below right] at (1) {$0$};
	\node [above right] at (2) {$1$};
	\node [below right] at (3) {$3$};
	\node [above right] at (4) {${69}$};
	\node [below right] at (5) {${86}$};
	\node [above right] at (6) {${93}$};
	\node [below right] at (7) {${77}$};
	\node [above right] at (8) {$5$};
	\node [below right] at (9) {${47}$};
	\node [above right] at (10) {${75}$};
	\node [below right] at (11) {${28}$};
	\node [above right] at (12) {${15}$};
	\node [below right] at (13) {${49}$};
	\end{tikzpicture}
\end{center}

\section{Vertex labelings of a difference graph with elements of a difference set}

In this section we make a digression on a (probably new) problem which is only seemingly unrelated to the main topic of this paper.
As a matter of fact, in the next section we will see how a specialization of this problem allows us to get several 2$-(v,\Gamma,\lambda)$
designs over $\F_q$ with $\Gamma$ of order $[v-1]_q$.

A $(v,\Gamma,\lambda)$ difference family in $G$ with only one base block will be naturally called a $(v,\Gamma,\lambda)$ {\it difference graph}. 
Anyway, we warn the reader that the term ``difference graph" already exists in other contexts
with a completely different meaning (see, e.g., \cite{Hammer}).
We note that when $G$ is cyclic and $\lambda=1$ this notion is completely equivalent to that of a {\it $\rho$-labeling of $\Gamma$} (see \cite{BE06}).
We also note that the vertex set of a $(v,\K_k,\lambda)$ difference graph in $G$ is nothing but a $(v,k,\lambda)$ {\it difference set} in $G$.
There is a wide literature on difference sets, for general background on them we refer to \cite{BJL,JPS}. 
Here we recall the definitions of the {\it Paley} and the {\it Singer} difference sets.
If $p\equiv3$ (mod 4) is a prime, then the set of non-zero squares of $\Z_p$ is a $(p,{p-1\over2},{p-3\over4})$ difference set,
which is called a Paley difference set. A Singer difference set is essentially 
the image of an arbitrary hyperplane of PG$(\F_q^v)$ in $[\Z_v]_q$. So its parameters are $([v]_q,[v-1]_q,[v-2]_q)$. 
Its development gives rise to the set of all the hyperplanes of PG$(\F_q^v)$, i.e., 
to the trivial complete 2$-(v,v-1,[v-2]_q)_q$ design.

The obvious necessary condition for the existence of a $(v,\Gamma,\lambda)$ difference graph in a certain group $G$ is that 
$\Gamma$ has size ${\lambda(v-1)\over2}$. If this condition is satisfied and $D$ is a $(v,k,\mu)$ difference set in $G$ for some 
pair $(\mu,k)$ with $k$ not smaller than the order of $\Gamma$ we can ask whether it is possible to realize the required difference 
graph in such a way that its vertex set is contained in $D$. In other words, we want to label the vertices of $\Gamma$ with elements of $D$ 
in such a way that the list of differences of adjacent labels covers every non-identity element of $G$ exactly $\lambda$ times. 
A labeling as above will be called a graceful $D$-labeling of $\Gamma$ since,  especially when $\lambda=1$, is reminiscent of 
the well known notion of a {\it graceful labeling} (see \cite{G} for a dynamic survey on this topic).

\begin{defn}\label{D-graceful}
Let $D$ be a $(v,k,\mu)$ difference set in a group $G$ and let $\Gamma=(V,E)$ be a graph
with $|V|\leq k$ and $|E|={\lambda(v-1)\over2}$ for some $\lambda$.
A {\it graceful $D$-labeling} of $\Gamma$ is an injective map $f:V\longrightarrow D$ such that the pair $(f(V),\{f(e) \ | \ e\in E\}$)
is a $(v,k,\lambda)$ difference graph. 
\end{defn}

We say that a pair $(D,\Gamma)$ as in the above definition is {\it admissible} or that $\Gamma$ is {\it $D$-admissible}.
Also, we say that $\Gamma$ is {\it $D$-graceful} if it admits a graceful $D$-labeling.
The problem of establishing which $D$-admissible graphs are\break $D$-graceful seems to us to be new. 
We speak of a graceful Singer or Paley ... $-$labeling of a graph $\Gamma$ to mean a
graceful $D$-labeling of $\Gamma$ with $D$ a difference set with the respective type.

Note that if $G$ is a group, then $G$ itself is trivially a $(v,v,v-2)$ difference set. Thus $\Gamma$ is $G$-graceful if and only if there exists a 
$(v,\Gamma,\lambda)$ difference graph in $G$. For instance, the well known fact that there is no $(43,7,1)$ difference set can be 
also expressed by saying that $\K_7$ is not $\Z_{43}$-graceful.

In the case that $\Gamma$ is complete, say $\Gamma=\K_h$, we also note that $\Gamma$ is $D$-graceful if and
only if that there exists a $(v,h,\lambda)$ difference set in $G$ which is contained in the $(v,k,\mu)$ difference set $D$.
Here is a remarkable example. The difference set $$D=\{{\bf1},3,{\bf5},6,7,{\bf11},17,18,20,21,{\bf24,25},26,{\bf27},29\}$$  is a Singer $(31,15,7)$ difference set in $\Z_{31}$ 
which contains the Singer $(31,6,1)$ difference set $D'=\{1,5,11,24,25,27\}$. Thus, considering that the development of $D$ is the set of
hyperplanes of PG$(\F_2^5)$ and that the development of $D'$ is the set of lines of PG$(\F_5^3)$, one might say that the 
projective plane of order 5 is ``nested" in the point-hyperplane design associated with the 4-dimensional projective space of order 2.
As far as we are aware nobody observed this before. Sophisticated ``games" using difference sets in the same group $G$, such as tiling $G$ with difference sets of the same parameters \cite{CKZ},  have been considered recently. Hence, it would be surprising if the problem of determining whether two difference sets
in the same group $G$ can be in inclusion relation was not considered before.

We could exhibit several examples of $D$-admissible graphs $\Gamma$ which are $G$-graceful but not $D$-graceful.
Consider for instance the $(15,7,3)$ difference set $D=\{0,1,2,4,5,8,10\}$ in $\Z_{15}$ and the $D$-admissible graph $\Gamma=C_3\cup C_4$ 
whose connected components are a 3-cycle and a 4-cycle.
There are many $(15,\Gamma,1)$ difference graphs; one of them is depicted below.
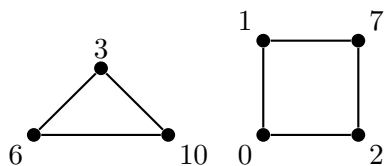
\begin{center}
	\begin{tikzpicture}[-,auto,node distance=2.5cm,thick,main node/.style={circle,fill=black,draw}]
	\node[circle,fill,scale=0.5](1) {};
	\node[circle,fill,scale=0.5, below left of=1](2) {};
	\node[circle,fill,scale=0.5, below right of=1](3) {};
	\node[circle,fill,scale=0.5, right of=3](4) {};
	\node[circle,fill,scale=0.5, above of=4](5) {};
	\node[circle,fill,scale=0.5, right of=4](7) {};
	\node[circle,fill,scale=0.5, above of=7](6) {};
		
	\path
	(1) edge node {} (2)
	(2) edge node {} (3)
	(3) edge node {} (1)
	(4) edge node {} (5)
	(5) edge node {} (6)
	(6) edge node {} (7)
	(7) edge node {} (4);
	
	\node [above] at (1) {3};
	\node [below left] at (2) {6};
	\node [below right] at (3) {10};
	\node [below left] at (4) {0};
	\node [above left] at (5) {1};
	\node [above right] at (6) {7};
	\node [below right] at (7) {2};
	\end{tikzpicture}
\end{center}
On the other hand, by exhaustive search we have checked that none of them has vertex set $D$. Thus $C_3\cup C_4$ is $\Z_{15}$-graceful but not $D$-graceful.

Here is instead an example of a $D$-admissible graph which is also $D$-graceful.
Let $D=\{1,4,5,6,7,9,11,16,17\}$ be the Paley $(19,9,4)$ difference set and let $\Gamma$ be the 3-prism.
A graceful Paley-labeling of $\Gamma$ is the following: 

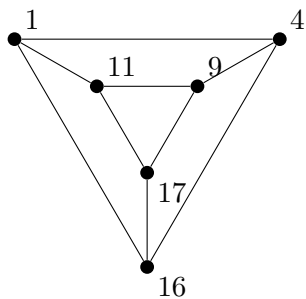
\begin{center}
	\begin{tikzpicture}
	\node[minimum size=1.5cm, regular polygon, regular polygon sides=3, rotate=180] (epta) {};
	\foreach \x in {1,2,3}{%
		\node[style={circle,fill,scale=0.5}] at (epta.corner \x) (e\x) {};
	}
	\node[minimum size=4cm, regular polygon, regular polygon sides=3, rotate=180] (septa) {};
	\foreach \x in {1,2,3}{%
		\node[style={circle,fill,scale=0.5}] at (septa.corner \x) (s\x) {};
	}
	\path
	(e1) edge node {} (e2)
	(e2) edge node {} (e3)
	(e3) edge node {} (e1)
	(s1) edge node {} (s2)
	(s2) edge node {} (s3)
	(s3) edge node {} (s1)
	(e1) edge node {} (s1)
	(e2) edge node {} (s2)
	(e3) edge node {} (s3)
	;
	\node [below right] at (e1) {17};
	\node [above right] at (e2) {9};
	\node [above right] at (e3) {11};
	\node [below right] at (s1) {16};
	\node [above right] at (s2) {4};
	\node [above right] at (s3) {1};
	\end{tikzpicture}
\end{center}

The notion introduced in Definition \ref{D-graceful} seems to be particularly interesting 
when $\Gamma$ has order $|D|$.
We do not have at the moment any example of an admissible pair $(D,\Gamma)$ where $\Gamma$ is
a regular connected graph of order $|D|$ that is not $D$-graceful. Thus we hazard the following conjecture.

\begin{conj}\label{conj}
Let $D$ be a $(v,k,\mu)$ difference set in $G$ and set $\lambda_i={\mu i\over\gcd(k-1,\mu)}$
for $1\leq i\leq \gcd(k-1,\mu)$. Then, any connected and regular graph $\Gamma$ of order $k$ and degree $(k-1)i\over\gcd(k-1,\mu)$
is $D$-graceful, hence there exists a $(v,\Gamma,\lambda_i)$ difference graph with vertex set $D$.
\end{conj}

The conjecture is trivially true in the extremal case $i=\gcd(k-1,\mu)$. Indeed in this case
$\Gamma$ should have degree $k-1$, hence it is necessarily the complete graph and 
the above statement says that there exists a $D$-graceful labeling of $\K_k$. This 
is equivalent to saying that $D$ is a difference set, which holds by assumption.
Thus, in particular, the conjecture is trivially true when $\gcd(k-1,\mu)=1$. 

The following proposition shows that the conjecture is true when $D$ is a Paley 
difference set and $\Gamma$ is {\it circulant}. For convenience of the reader we recall that
if $S$ is a subset of $\{1,\dots,\lfloor{n\over2}\rfloor\}$, then the {\it circulant graph} $C(\Z_n;S)$ 
is the graph with vertex set $\Z_n$ whose edges are all pairs of the form $\{x,x+s\}$
with $x\in \Z_n$ and $s\in S$.

\begin{prop}\label{Paley}
Let $D$ be the Paley $(4n+3,2n+1,n)$ difference set and let $\Gamma$ be a circulant
graph of order $2n+1$. Then $\Gamma$ is $D$-graceful.
\end{prop}
\begin{proof}
By assumption, $4n+3$ is a prime and $D$ is the set of non-zero squares of $\Z_{4n+3}$, i.e., the
subgroup $D$ of order $2n+1$ of the multiplicative group of $\Z_{4n+3}$. Also by assumption we
have $\Gamma=C(\Z_{2n+1};S)$ for a suitable set $S$. We claim that 
any isomorphism $f$ between the two groups $(\Z_{2n+1},+)$ and $(D,\cdot)$ is a $D$-graceful labeling of $\Gamma$. 
It is enough to show that the list of differences, in $\Z_{4n+3}$, of the graph $\Gamma':=f(\Gamma)$
is evenly distributed over the non-zero elements of $\Z_{4n+3}$.
The edge-set of $\Gamma'$ consists of all possible pairs of the form $\{d,ds'\}$ with $d\in D$ and $s'\in S':=f(S)$.
Thus we have $$\Delta\Gamma'=\{1,-1\}\cdot\{d(1-s') \ | \ d\in D,s'\in S'\}=\{1,-1\}\cdot D\cdot\{1-s' \ | \ s'\in S'\}.$$
Now recall that $4n+3$ is a prime and that $-1$ is a non-square in every field of order congruent to 3 (mod 4).
It follows that we have $\{1,-1\}\cdot D=\Z_{4n+3}^*$. Hence we see that $\Delta\Gamma'$ covers every non-zero 
element of $\Z_{4n+3}$ exactly $|S|$ times and the assertion follows.
\end{proof}

\section{Singer graceful graphs and related graph decompositions over a finite field}\label{Singersection}

By Definition \ref{D-graceful} and Theorem \ref{df_q}, we can state the following.
\begin{prop}\label{singergraceful}
If a graph $\Gamma$ of order $[v-1]_q$ 
and size ${\lambda q[v-1]_q\over2}$ is Singer-graceful, then there exists 
a cyclic $2$$-(v,\Gamma,\lambda)$ design over $\F_q$. 
\end{prop}

Thus, if Conjecture \ref{conj} were true, we would have an infinite family of non-trivial graph decompositions over a finite field.
Indeed, specializing the conjecture to the case that $D$ is the Singer $([v]_q,[v-1]_q,[v-2]_q)$ difference set, we note that $\lambda_i=i$
for each $i$ and hence  we obtain the following subconjecture.

\begin{conj}\label{subconj}
Any regular graph $\Gamma$ of order $[v-1]_q$ and degree $qi$ with $1\leq i\leq [v-2]_q$ is Singer-graceful,
hence there exists a cyclic $2$-$(v,\Gamma,i)$ design over $\F_q$.
\end{conj}

Now we give some examples supporting the above subconjecture; namely a graceful Singer-labeling of some cycles, of a prism, of a
Moebius ladder and of some generalized Petersen graphs.

\subsection{A cyclic $2$$-(v,C_{v-1},1)_2$ design for $v=4,5,6,7$}\label{singergracefulcycles}

By Proposition \ref{singergraceful}, to prove the existence of a cyclic 2$-(v,C_{v-1},1)_2$ design it is enough to show a
Singer-graceful labeling of $[C_{v-1}]_2$, the cycle of length $[v-1]_2$. We do this for $v=4,5,6,7$.

$\bf v=4$.

\noindent
The image of a Singer $(15,7,3)$ difference set in $\Z_{15}$ is
$$D_{15}=\{0,1,2,4,5,8,10\}.$$
and a graceful $D_{15}$-labeling of $C_{7}$ is given by
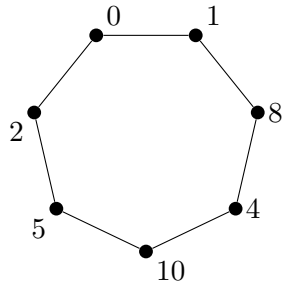
\begin{center}
	\begin{tikzpicture}
	\node[minimum size=3cm, regular polygon, regular polygon sides=7, rotate=180] (epta) {};
	\foreach \x in {1,2,...,7}{%
		\node[style={circle,fill,scale=0.5}] at (epta.corner \x) (e\x) {};
	}
	\path
	(e1) edge node {} (e2)
	(e2) edge node {} (e3)
	(e3) edge node {} (e4)
	(e4) edge node {} (e5)
	(e5) edge node {} (e6)
	(e6) edge node {} (e7)
	(e7) edge node {} (e1)
	;
	
	\node [below right] at (e1) {10};
	\node [right] at (e2) {4};
	\node [right] at (e3) {8};
	\node [above right] at (e4) {1};
	\node [above right] at (e5) {0};
	\node [below left] at (e6) {2};
	\node [below left] at (e7) {5};
	\end{tikzpicture}
\end{center}

$\bf v=5$.

\noindent
The image of a Singer $(31,15,7)$ difference set in $\Z_{31}$ is
$$D_{31}=\{1,2,3,4,6,8,12,15,16,17,23,24,27,29,30\}$$
and a graceful $D_{31}$-labeling of $C_{15}$ is given by
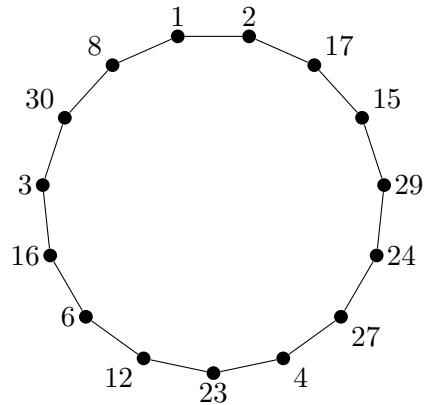
\begin{center}
	\begin{tikzpicture}
	\node[minimum size=4.5cm, regular polygon, regular polygon sides=15, rotate=180] (epta) {};
	\foreach \x in {1,2,...,15}{%
		\node[style={circle,fill,scale=0.5}] at (epta.corner \x) (e\x) {};
	}
	\path
	(e1) edge node {} (e2)
	(e2) edge node {} (e3)
	(e3) edge node {} (e4)
	(e4) edge node {} (e5)
	(e5) edge node {} (e6)
	(e6) edge node {} (e7)
	(e7) edge node {} (e8)
	(e8) edge node {} (e9)
	(e9) edge node {} (e10)
	(e10) edge node {} (e11)
	(e11) edge node {} (e12)
	(e12) edge node {} (e13)
	(e13) edge node {} (e14)
	(e14) edge node {} (e15)
	(e15) edge node {} (e1)
	;
	
	\node [above] at (e9) {1};
	\node [above] at (e8) {2};
	\node [above right] at (e7) {17};
	\node [above right] at (e6) {15};
	\node [right] at (e5) {29};
	\node [right] at (e4) {24};
	\node [below right] at (e3) {27};
	\node [below right] at (e2) {4};
	\node [below] at (e1) {23};
	\node [below left] at (e15) {12};
	\node [left] at (e14) {6};
	\node [left] at (e13) {16};
	\node [left] at (e12) {3};
	\node [above left] at (e11) {30};
	\node [above left] at (e10) {8};
	\end{tikzpicture}
\end{center}

$\bf v=6$.

\noindent
The image of a Singer $(63,31,15)$ difference set in $\Z_{63}$ is

\medskip\noindent
$D_{63}=\{0,1,2,3,4,6,7,8,9,12,13,14,16,18,19,24,26,27,28,$

\medskip
\hfill$32,33,35,36,38,41,45,48,49,52,54,56\}$

\medskip\noindent
and a graceful $D_{63}$-labeling of $C_{31}$ is given by
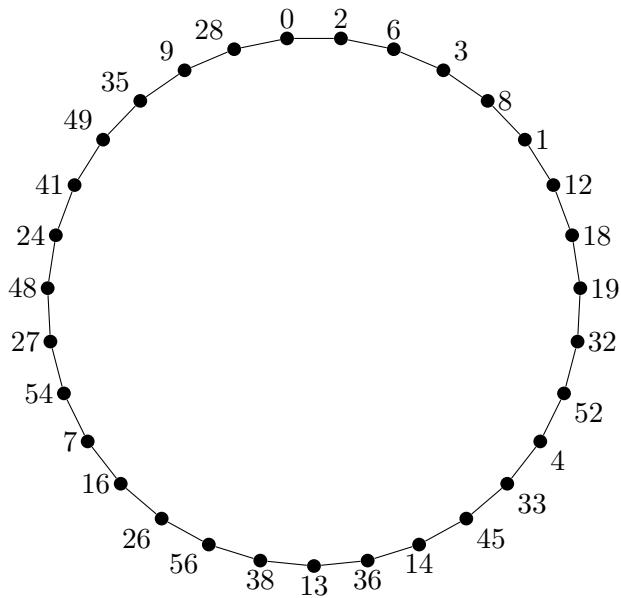
\begin{center}
	\begin{tikzpicture}\node[minimum size=7cm, regular polygon, regular polygon sides=31, rotate=180] (epta) {};
	\foreach \x in {1,2,...,31}{%
		\node[style={circle,fill,scale=0.5}] at (epta.corner \x) (e\x) {};
	}
	\path
	(e1) edge node {} (e2)
	(e2) edge node {} (e3)
	(e3) edge node {} (e4)
	(e4) edge node {} (e5)
	(e5) edge node {} (e6)
	(e6) edge node {} (e7)
	(e7) edge node {} (e8)
	(e8) edge node {} (e9)
	(e9) edge node {} (e10)
	(e10) edge node {} (e11)
	(e11) edge node {} (e12)
	(e12) edge node {} (e13)
	(e13) edge node {} (e14)
	(e14) edge node {} (e15)
	(e15) edge node {} (e16)
	(e16) edge node {} (e17)
	(e17) edge node {} (e18)
	(e18) edge node {} (e19)
	(e19) edge node {} (e20)
	(e20) edge node {} (e21)
	(e21) edge node {} (e22)
	(e22) edge node {} (e23)
	(e23) edge node {} (e24)
	(e24) edge node {} (e25)
	(e25) edge node {} (e26)
	(e26) edge node {} (e27)
	(e27) edge node {} (e28)
	(e28) edge node {} (e29)
	(e29) edge node {} (e30)
	(e30) edge node {} (e31)
	(e31) edge node {} (e1)
	;
	
	\node [above] at (e17) {0};
	\node [above] at (e16) {2};
	\node [above] at (e15) {6};
	\node [above right] at (e14) {3};
	\node [right] at (e13) {8};
	\node [right] at (e12) {1};
	\node [right] at (e11) {12};
	\node [right] at (e10) {18};
	\node [right] at (e9) {19};
	\node [right] at (e8) {32};
	\node [below right] at (e7) {52};
	\node [below right] at (e6) {4};
	\node [below right] at (e5) {33};
	\node [below right] at (e4) {45};
	\node [below] at (e3) {14};
	\node [below] at (e2) {36};
	\node [below] at (e1) {13};
	\node [below] at (e31) {38};
	\node [below left ] at (e30) {56};
	\node [below left] at (e29) {26};
	\node [left] at (e28) {16};
	\node [left] at (e27) {7};
	\node [left] at (e26) {54};
	\node [left] at (e25) {27};
	\node [left] at (e24) {48};
	\node [left] at (e23) {24};
	\node [left] at (e22) {41};
	\node [above left] at (e21) {49};
	\node [above left] at (e20) {35};
	\node [above left] at (e19) {9};
	\node [above left] at (e18) {28};
	\end{tikzpicture}
\end{center}

$\bf v=7$.

The image of a Singer $(127,63,31)$ difference set in $\Z_{127}$ is

\medskip\noindent
\small $D_{127}=\{1, 2, 3, 4, 5, 6, 8, 9, 10, 11, 12, 15, 16, 17, 18, 20, 22, 23, 24, 29, 30,
32, 33, 34, 36,$ 

\noindent
$39, 40, 44, 46, 48, 49, 55, 57, 58, 59, 60, 64, 65, 66, 68,69, 71, 72, 75, 78, 80, 83, 88, 91, 92,$

\noindent\hfill
$93, 96, 98, 99, 101, 105, 109, 110,113, 114, 116, 118, 120\}.$

\normalsize\medskip
The above difference set $D_{127}$ admits 2 as a multiplier. In this case we have $2\cdot D_{127}=D_{127}$, i.e., $D_{127}$
is fixed by the multiplication by 2. The task of finding a graceful $D_{127}$-labeling of $C_{63}$ 
is facilitated if we impose that the resultant $(127,C_{63},1)$ difference graph is also fixed by the multiplication by 2.
This happens provided that the multiplication by 2 acts as a rotation about the center of this graph by ${2\pi\over 7}$.
Equivalently, the required difference graph has to be the union of the orbits of a path of size 9 under Frob$([\Z_7]_2)$.
A solution is represented in the next figure. The reader can recognize that the resulting cycle is indeed
the union of the orbits of the path $[1,3,9,101,91,5,83,113,11,2]$ under Frob$([\Z_7]_2)$.

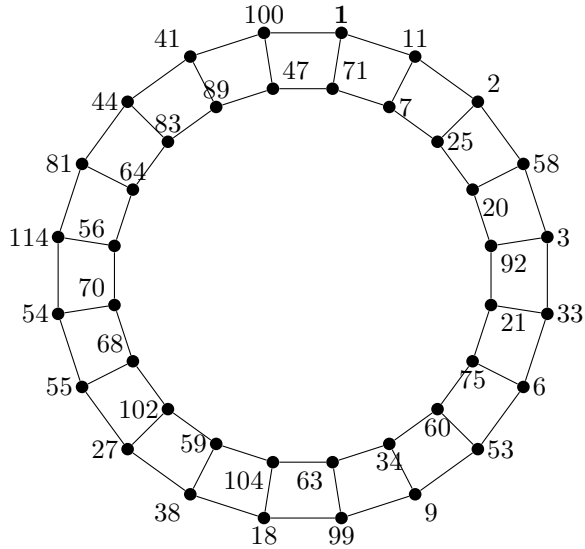
\scriptsize\begin{center}
	\begin{tikzpicture}\node[minimum size=11cm, regular polygon, regular polygon sides=63, rotate=180] (epta) {};
	\foreach \x in {1,2,...,63}{%
		\node[style={circle,fill,scale=0.5}] at (epta.corner \x) (e\x) {};
	}
	
	\path
	(e1) edge node {} (e2)
	(e2) edge node {} (e3)
	(e3) edge node {} (e4)
	(e4) edge node {} (e5)
	(e5) edge node {} (e6)
	(e6) edge node {} (e7)
	(e7) edge node {} (e8)
	(e8) edge node {} (e9)
	(e9) edge node {} (e10)
	(e10) edge node {} (e11)
	(e11) edge node {} (e12)
	(e12) edge node {} (e13)
	(e13) edge node {} (e14)
	(e14) edge node {} (e15)
	(e15) edge node {} (e16)
	(e16) edge node {} (e17)
	(e17) edge node {} (e18)
	(e18) edge node {} (e19)
	(e19) edge node {} (e20)
	(e20) edge node {} (e21)
	(e21) edge node {} (e22)
	(e22) edge node {} (e23)
	(e23) edge node {} (e24)
	(e24) edge node {} (e25)
	(e25) edge node {} (e26)
	(e26) edge node {} (e27)
	(e27) edge node {} (e28)
	(e28) edge node {} (e29)
	(e29) edge node {} (e30)
	(e30) edge node {} (e31)
	(e31) edge node {} (e32)
	(e32) edge node {} (e33)
	(e33) edge node {} (e34)
	(e34) edge node {} (e35)
	(e35) edge node {} (e36)
	(e36) edge node {} (e37)
	(e37) edge node {} (e38)
	(e38) edge node {} (e39)
	(e39) edge node {} (e40)
	(e40) edge node {} (e41)
	(e41) edge node {} (e42)
	(e42) edge node {} (e43)
	(e43) edge node {} (e44)
	(e44) edge node {} (e45)
	(e45) edge node {} (e46)
	(e46) edge node {} (e47)
	(e47) edge node {} (e48)
	(e48) edge node {} (e49)
	(e49) edge node {} (e50)
	(e50) edge node {} (e51)
	(e51) edge node {} (e52)
	(e52) edge node {} (e53)
	(e53) edge node {} (e54)
	(e54) edge node {} (e55)
	(e55) edge node {} (e56)
	(e56) edge node {} (e57)
	(e57) edge node {} (e58)
	(e58) edge node {} (e59)
	(e59) edge node {} (e60)
	(e60) edge node {} (e61)
	(e61) edge node {} (e62)
	(e62) edge node {} (e63)
	(e63) edge node {} (e1)
	;
	
	\node [above] at (e33) {1};
	\node [above] at (e32) {3};
	\node [above] at (e31) {9};
	\node [above] at (e30) {101};
	\node [above right] at (e29) {91};
	\node [above right] at (e28) {5};
	\node [above right] at (e27) {83};
	\node [above right] at (e26) {113};
	\node [above right] at (e25) {11};
	\node [above right] at (e24) {2};
	\node [above right] at (e23) {6};
	\node [above right] at (e22) {18};
	\node [above right] at (e21) {75};
	\node [above right] at (e20) {55};
	\node [above right] at (e19) {10};
	\node [above right] at (e18) {39};
	\node [right] at (e17) {99};
	\node [right] at (e16) {22};
	\node [right] at (e15) {4};
	\node [right] at (e14) {12};
	\node [right] at (e13) {36};
	\node [right] at (e12) {23};
	\node [right] at (e11) {110};
	\node [right] at (e10) {20};
	\node [below right] at (e9) {78};
	\node [below right] at (e8) {71};
	\node [below right] at (e7) {44};
	\node [below] at (e6) {8};
	\node [below] at (e5) {24};
	\node [below] at (e4) {72};
	\node [below] at (e3) {46};
	\node [below] at (e2) {93};
	\node [below] at (e1) {40};
	\node [below] at (e63) {29};
	\node [below] at (e62) {15};
	\node [below] at (e61) {88};
	\node [below left] at (e60) {16};
	\node [below left] at (e59) {48};
	\node [below  left] at (e58) {17};
	\node [below left] at (e57) {92};
	\node [left] at (e56) {59};
	\node [left] at (e55) {80};
	\node [left] at (e54) {58};
	\node [left] at (e53) {30};
	\node [left] at (e52) {49};
	\node [left] at (e51) {32};
	\node [left] at (e50) {96};
	\node [left] at (e49) {34};
	\node [left] at (e48) {57};
	\node [left] at (e47) {118};
	\node [left] at (e46) {33};
	\node [left] at (e45) {116};
	\node [left] at (e44) {60};
	\node [left] at (e43) {98};
	\node [left] at (e42) {64};
	\node [left] at (e41) {65};
	\node [above left] at (e40) {68};
	\node [above left] at (e39) {114};
	\node [above left] at (e38) {109};
	\node [above left] at (e37) {66};
	\node [above left] at (e36) {105};
	\node [above] at (e35) {120};
	\node [above] at (e34) {69};
	\end{tikzpicture}
\end{center}

\normalsize
\subsubsection{A prism decomposition over $\F_3$}\label{prismsection}
The {\it prism graph} on $2n$ vertices is the cubic graph corresponding to the skeleton of an $n$-prism. 
Following \cite{HMP} we denote it by $\Pi_n$. 

Let us construct a cyclic $2$$-$$(5,\Pi_{40},1)$ design over $\F_3$. 
For this, it is enough to give a Singer-graceful labeling of $\Pi_{40}$.
The image of a Singer $(121,40,13)$ difference set in $\Z_{121}$ is

\medskip\noindent\small
$D = \{1,2,3,6,7,9,11,18,20,21,25,27,33,34,38,41,44,47,53,54,55,56,$

\smallskip
\hfill$58,59,60,63,64,68,70,71,75,81,83,89,92,99,100,102,104,114\}$
\medskip

\normalsize\noindent
and a $D$-graceful labeling of $\Pi_{40}$ is the following:

\small\begin{center}
\begin{tikzpicture}
\node[minimum size=5cm, regular polygon, regular polygon sides=20, rotate=180] (epta) {};
\foreach \x in {1,...,20}{%
\node[circle,fill,scale=0.5] at (epta.corner \x) (e\x) {};
}
\node[minimum size=6.5cm, regular polygon, regular polygon sides=20, rotate=180] (septa) {};
\foreach \x in {1,...,20}{%
\node[circle,fill,scale=0.5] at (septa.corner \x) (s\x) {};
}
\path
(s1) edge node {} (s2)
(s2) edge node {} (s3)
(s3) edge node {} (s4)
(s4) edge node {} (s5)
(s5) edge node {} (s6)
(s6) edge node {} (s7)
(s7) edge node {} (s8)
(s8) edge node {} (s9)
(s9) edge node {} (s10)
(s10) edge node {} (s11)
(s11) edge node {} (s12)
(s12) edge node {} (s13)
(s13) edge node {} (s14)
(s14) edge node {} (s15)
(s15) edge node {} (s16)
(s16) edge node {} (s17)
(s17) edge node {} (s18)
(s18) edge node {} (s19)
(s19) edge node {} (s20)
(s20) edge node {} (s1)

(s1) edge node {} (e1)
(s2) edge node {} (e2)
(s3) edge node {} (e3)
(s4) edge node {} (e4)
(s5) edge node {} (e5)
(s6) edge node {} (e6)
(s7) edge node {} (e7)
(s8) edge node {} (e8)
(s9) edge node {} (e9)
(s10) edge node {} (e10)
(s11) edge node {} (e11)
(s12) edge node {} (e12)
(s13) edge node {} (e13)
(s14) edge node {} (e14)
(s15) edge node {} (e15)
(s16) edge node {} (e16)
(s17) edge node {} (e17)
(s18) edge node {} (e18)
(s19) edge node {} (e19)
(s20) edge node {} (e20)

(e1) edge node {} (e2)
(e2) edge node {} (e3)
(e3) edge node {} (e4)
(e4) edge node {} (e5)
(e5) edge node {} (e6)
(e6) edge node {} (e7)
(e7) edge node {} (e8)
(e8) edge node {} (e9)
(e9) edge node {} (e10)
(e10) edge node {} (e11)
(e11) edge node {} (e12)
(e12) edge node {} (e13)
(e13) edge node {} (e14)
(e14) edge node {} (e15)
(e15) edge node {} (e16)
(e16) edge node {} (e17)
(e17) edge node {} (e18)
(e18) edge node {} (e19)
(e19) edge node {} (e20)
(e20) edge node {} (e1);

\node [above right] at (e11) {71};
\node [right] at (e10) {7};
\node [right ] at (e9) {25};
\node [below right] at (e8) {20};
\node [below right] at (e7) {92};
\node [below right] at (e6) {21};
\node [below ] at (e5) {75};
\node [below] at (e4) {60};
\node [below] at (e3) {34};
\node [below left] at (e2) {63};
\node [below left] at (e1) {104};
\node [left] at (e20) {59};
\node [left] at (e19) {102};
\node [above left] at (e18) {68};
\node [above left] at (e17) {70};
\node [above left] at (e16) {56};
\node [above] at (e15) {64};
\node [above] at (e14) {83};
\node [above] at (e13) {89};
\node [above right] at (e12) {47};

\node [above] at (s11) {\bf1};
\node [above] at (s10) {11};
\node [above right ] at (s9) {2};
\node [right] at (s8) {58};
\node [right] at (s7) {3};
\node [right] at (s6) {33};
\node [right] at (s5) {6};
\node [right] at (s4) {53};
\node [below right] at (s3) {9};
\node [below ] at (s2) {99};
\node [below ] at (s1) {18};
\node [below left] at (s20) {38};
\node [left] at (s19) {27};
\node [left] at (s18) {55};
\node [left] at (s17) {54};
\node [left] at (s16) {114};
\node [left] at (s15) {81};
\node [left] at (s14) {44};
\node [above left] at (s13) {41};
\node [above] at (s12) {100};
\end{tikzpicture}
\end{center}

\normalsize
Note that the above graph is fixed by the multiplication by 3. Indeed, denoting by $v_i$ the label of the $i$-th
``outer" vertex and by $v'_i$ the label of the corresponding ``inner" vertex, one can check that $v_{i+4}\equiv3v_i$ (mod 121) and 
$v'_{i+4}\equiv3v'_i$ (mod 121) for every possible $i$ (it is understood that the indices have to be considered modulo 40 and that 
the $(i+1)$-th vertex of the outer cycle follows the $i$-th one clockwise).
This means that the multiplication by 3 corresponds to a clockwise rotation of the above graph about its center by 72 degrees. 

\subsubsection{Generalized Petersen decompositions over $\F_3$}

Let $n\geq5$ be an integer, let $\Z'_n=\{0',1',\dots,(n-1)'\}$ be a disjoint isomorphic copy of $\Z_n$, and let
$k$ be an integer in the closed interval $[2,n-2]$. The {\it generalized Petersen graph} $P(n,k)$
is the graph of order $2n$ with vertex set $\Z_n \ \cup \ \Z'_n$ which is the
union of the circulant graphs $C(\Z_n;\{\pm1\})$, $C(\Z'_n;\{\pm k\})$ and the perfect matching $\{\{i,i'\} \ | \ 0\leq i\leq n-1\}$.   

We want to construct a cyclic 2$-(5,P(20,k),1)$ design over $\F_3$ for each possible $k$, hence for $k=2,3,4,5,6,7,8,9$. 
For this, it is enough to give a $D$-graceful labeling of $P(20,k)$ where $D$ is the Singer $(121,40,13)$ difference set that we gave in the previous subsection. 
Here is, for instance, a Singer-labeling of $P(20,3)$:

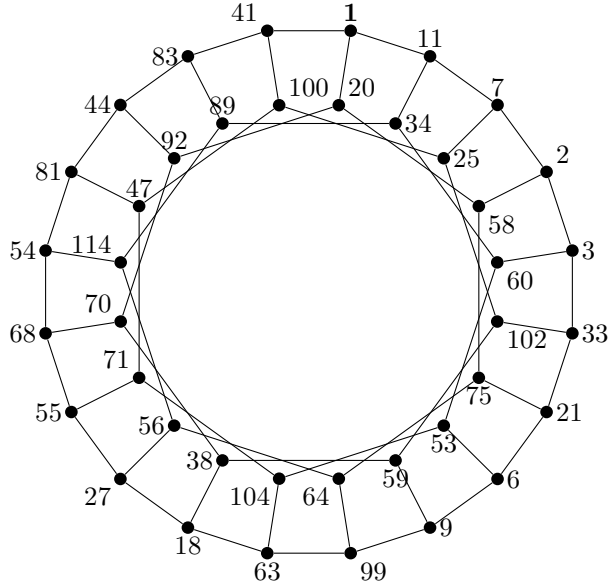
\small\begin{center}
\begin{tikzpicture}
\node[minimum size=5cm, regular polygon, regular polygon sides=20, rotate=180] (epta) {};
\foreach \x in {1,...,20}{%
\node[circle,fill,scale=0.5] at (epta.corner \x) (e\x) {};
}
\node[minimum size=7cm, regular polygon, regular polygon sides=20, rotate=180] (septa) {};
\foreach \x in {1,...,20}{%
\node[circle,fill,scale=0.5] at (septa.corner \x) (s\x) {};
}
\path
(s1) edge node {} (s2)
(s2) edge node {} (s3)
(s3) edge node {} (s4)
(s4) edge node {} (s5)
(s5) edge node {} (s6)
(s6) edge node {} (s7)
(s7) edge node {} (s8)
(s8) edge node {} (s9)
(s9) edge node {} (s10)
(s10) edge node {} (s11)
(s11) edge node {} (s12)
(s12) edge node {} (s13)
(s13) edge node {} (s14)
(s14) edge node {} (s15)
(s15) edge node {} (s16)
(s16) edge node {} (s17)
(s17) edge node {} (s18)
(s18) edge node {} (s19)
(s19) edge node {} (s20)
(s20) edge node {} (s1)

(s1) edge node {} (e1)
(s2) edge node {} (e2)
(s3) edge node {} (e3)
(s4) edge node {} (e4)
(s5) edge node {} (e5)
(s6) edge node {} (e6)
(s7) edge node {} (e7)
(s8) edge node {} (e8)
(s9) edge node {} (e9)
(s10) edge node {} (e10)
(s11) edge node {} (e11)
(s12) edge node {} (e12)
(s13) edge node {} (e13)
(s14) edge node {} (e14)
(s15) edge node {} (e15)
(s16) edge node {} (e16)
(s17) edge node {} (e17)
(s18) edge node {} (e18)
(s19) edge node {} (e19)
(s20) edge node {} (e20)

(e1) edge node {} (e18)
(e2) edge node {} (e19)
(e3) edge node {} (e20)
(e4) edge node {} (e1)
(e5) edge node {} (e2)
(e6) edge node {} (e3)
(e7) edge node {} (e4)
(e8) edge node {} (e5)
(e9) edge node {} (e6)
(e10) edge node {} (e7)
(e11) edge node {} (e8)
(e12) edge node {} (e9)
(e13) edge node {} (e10)
(e14) edge node {} (e11)
(e15) edge node {} (e12)
(e16) edge node {} (e13)
(e17) edge node {} (e14)
(e18) edge node {} (e15)
(e19) edge node {} (e16)
(e20) edge node {} (e17);

\node [above right] at (e11) {20};
\node [right] at (e10) {34};
\node [right ] at (e9) {25};
\node [below right] at (e8) {58};
\node [below right] at (e7) {60};
\node [below right] at (e6) {102};
\node [below ] at (e5) {75};
\node [below] at (e4) {53};
\node [below] at (e3) {59};
\node [below left] at (e2) {64};
\node [below left] at (e1) {104};
\node [left] at (e20) {38};
\node [left] at (e19) {56};
\node [above left] at (e18) {71};
\node [above left] at (e17) {70};
\node [above left] at (e16) {114};
\node [above] at (e15) {47};
\node [above] at (e14) {92};
\node [above] at (e13) {89};
\node [above right] at (e12) {100};

\node [above] at (s11) {\bf1};
\node [above] at (s10) {11};
\node [above] at (s9) {7};
\node [above right ] at (s8) {2};
\node [right] at (s7) {3};
\node [right] at (s6) {33};
\node [right] at (s5) {21};
\node [right] at (s4) {6};
\node [right] at (s3) {9};
\node [below right] at (s2) {99};
\node [below ] at (s1) {63};
\node [below ] at (s20) {18};
\node [below left] at (s19) {27};
\node [left] at (s18) {55};
\node [left] at (s17) {68};
\node [left] at (s16) {54};
\node [left] at (s15) {81};
\node [left] at (s14) {44};
\node [left] at (s13) {83};
\node [above left] at (s12) {41};
\end{tikzpicture}
\end{center}

\normalsize
Also here, as for the prism seen below, the multiplication by 3 corresponds to a clockwise rotation of the graph about its center by 72 degrees. 
We report below how to label the eight vertices $0,1,2,3,0',1',2',3'$ of $P(20,k)$ 
to obtain, with the same method, a $D$-graceful labeling of $P(20,k)$ for $2\leq k\leq9$.

\small\medskip\begin{center}
\renewcommand\arraystretch{1.3} 
\begin{tabular}{|c|c|c|c|c|c|c|c|c|c|c|c|}
\hline $k$ & $\{v_0,v_1,v_2,v_3\}$ & $\{v'_0,v'_1,v'_2,v'_3\}$\\
\hline $2$ & $\{1,7,2,34\}$ & $\{70,20,11,53\}$ \\
\hline 3 & $\{1,11,7,2\}$ & $\{20,34,25,58\}$\\
\hline 4 & $\{1,7,2,100\}$ & $\{11,20,64,25\}$ \\
\hline 5 & $\{1,20,2,25\}$ & $\{64,7,53,11\}$ \\
\hline
\end{tabular}\hfill
\begin{tabular}{|c|c|c|c|c|c|c|c|c|c|c|c|}
\hline $k$ & $\{v_0,v_1,v_2,v_3\}$ & $\{v'_0,v'_1,v'_2,v'_3\}$\\
\hline $6$ & $\{1,58,7,25\}$ & $\{11,2,20,34\}$ \\ 
\hline 7& $\{1,25,7,34\}$ & $\{20,114,2,11\}$\\
\hline 8 & $\{1,20,58,7\}$ & $\{25,2,11,34\}$ \\
\hline 9 & $\{1,53,11,64\}$ & $\{7,2,20,25\}$ \\
\hline
\end{tabular}
\end{center}
\normalsize
As a matter of fact it would not have been necessary to treat both $k=3$ and $k=7$ since $P(20,3)$ and $P(20,7)$ are  isomorphic (see, e.g., \cite{SS}).

In conclusion, we have given a $(5,P(20,k),1)$ difference graph over $\F_3$ which is fixed by Frob$[\Z_5]_3$ for $2\leq k\leq 9$.

\subsubsection{A Moebius ladder decomposition over $\F_3$}
The Moebius ladder $M_{2n}$ is the circulant graph $C(\Z_{2n};\{1,n\})$. In simpler words, it is the
cubic graph of order $2n$ obtained from the $2n$-cycle $(c_0,c_1,\dots,c_{2n-1})$ by
adding all possible diameters, i.e., all edges of the form $\{c_i,c_{i+n}\}$ with $1\leq i\leq n$. 

Let us construct a cyclic 2$-(5,M_{40},1)$ design over $\F_3$. For this, it is enough to give a $D$-graceful labeling
of $M_{40}$ where $D$ is the Singer $(121,40,13)$ difference set that we gave in Subsection \ref{prismsection}. 
Such a labeling is the one shown below:

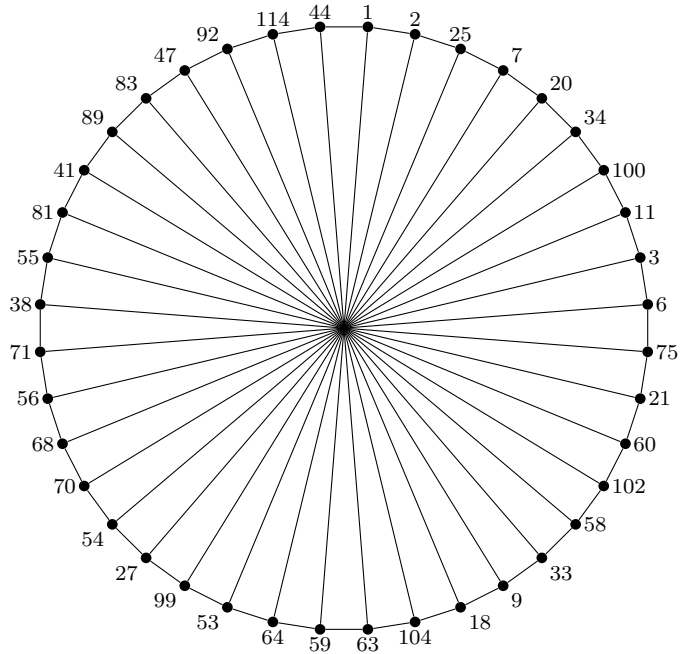
\scriptsize\begin{center}
	\begin{tikzpicture}\node[minimum size=8cm, regular polygon, regular polygon sides=40, rotate=180] (epta) {};
	\foreach \x in {1,...,40}{%
		\node[circle,fill,scale=0.5] at (epta.corner \x) (\x) {};
	}	
	\path
	(1) edge node {} (2)
	(2) edge node {} (3)
	(3) edge node {} (4)
	(4) edge node {} (5)
	(5) edge node {} (6)
	(6) edge node {} (7)
	(7) edge node {} (8)
	(8) edge node {} (9)
	(9) edge node {} (10)
	(10) edge node {} (11)
	(11) edge node {} (12)
	(12) edge node {} (13)
	(13) edge node {} (14)
	(14) edge node {} (15)
	(15) edge node {} (16)
	(16) edge node {} (17)
	(17) edge node {} (18)
	(18) edge node {} (19)
	(19) edge node {} (20)
	(20) edge node {} (21)
	(21) edge node {} (22)
	(22) edge node {} (23)
	(23) edge node {} (24)
	(24) edge node {} (25)
	(25) edge node {} (26)
	(26) edge node {} (27)
	(27) edge node {} (28)
	(28) edge node {} (29)
	(29) edge node {} (30)
	(30) edge node {} (31)
	(31) edge node {} (32)
	(32) edge node {} (33)
	(33) edge node {} (34)
	(34) edge node {} (35)	
	(35) edge node {} (36)	
	(36) edge node {} (37)	
	(37) edge node {} (38)	
	(38) edge node {} (39)	
	(39) edge node {} (40)	
	(40) edge node {} (1)
	(1) edge node {} (21)
	(2) edge node {} (22)
	(3) edge node {} (23)
	(4) edge node {} (24)
	(5) edge node {} (25)
	(6) edge node {} (26)
	(7) edge node {} (27)
	(8) edge node {} (28)
	(9) edge node {} (29)
	(10) edge node {} (30)
	(11) edge node {} (31)
	(12) edge node {} (32)
	(13) edge node {} (33)
	(14) edge node {} (34)
	(15) edge node {} (35)
	(16) edge node {} (36)
	(17) edge node {} (37)
	(18) edge node {} (38)
	(19) edge node {} (39)
	(20) edge node {} (40)
	;
	
	\node [below] at (1) {59};
	\node [below] at (40) {64};
	\node [below left] at (39) {53};
	\node [below left] at (38) {99};
	\node [below left] at (37) {27};
	\node [below left] at (36) {54};
	\node [left] at (35) {70};
	\node [left] at (34) {68};
	\node [left] at (33) {56};
	\node [left] at (32) {71};
	\node [left] at (31) {38};
	\node [left] at (30) {55};
	\node [left] at (29) {81};
	\node [left] at (28) {41};
	\node [above left] at (27) {89};
	\node [above left] at (26) {83};
	\node [above left] at (25) {47};
	\node [above left] at (24) {92};
	\node [above ] at (23) {114};
	\node [above ] at (22) {44};
	\node [above] at (21) {1};
	\node [above] at (20) {2};
	\node [above] at (19) {25};
	\node [above right] at (18) {7};
	\node [above right] at (17) {20};
	\node [above right] at (16) {34};
	\node [right] at (15) {100};
	\node [right] at (14) {11};
	\node [right] at (13) {3};
	\node [right] at (12) {6};
	\node [right] at (11) {75};
	\node [right] at (10) {21};
	\node [right] at (9) {60};
	\node [right] at (8) {102};
	\node [right] at (7) {58};
	\node [below right] at (6) {33};
	\node [below right] at (5) {9};
	\node [below right] at (4) {18};
	\node [below ] at (3) {104};
	\node [below ] at (2) {63};
	
	\end{tikzpicture}
\end{center}

\normalsize
The reader can recognize, once again, that the $(121,M_{40},1)$ difference graph constructed above
is fixed by Frob$[\Z_5]_3$. The external 40-cycle is obtainable by joining the orbits of the
path $[1,2,25,7,20,34,100,11,3]$ under this group.

\subsection{Near resolvable $(v,2,1)_q$ designs}

We recall that a {\it resolvable} 2$-(v,k,\lambda)$ design is a triple $({\cal P},{\cal B},{\cal R})$ where
$({\cal P},{\cal B})$ is a 2$-(v,k,\lambda)$ design and ${\cal R}$ is a partition of ${\cal B}$ into classes
({\it parallel classes}) each of which is, in its turn, a partition of ${\cal P}$.

We also recall that a {\it near resolvable} 2$-(v,k,k-1)$ design is a triple $({\cal P},{\cal B},{\cal R})$ where
$({\cal P},{\cal B})$ is a 2$-(v,k,k-1)$ design and ${\cal R}$ is a partition of ${\cal B}$ into classes
({\it near parallel classes}) each of which gives a partition of all points except one.

A 2$-(v,2,1)$ design is nothing but the complete graph $\K_v$ and its $q$-analog, namely a 2$-(v,2,1)_q$ design, is the 
point-line design of PG$(\F_q^v)$. These designs are clearly trivial. A resolvable 2$-(v,2,1)$ design (more commonly known
as a {\it one-factorization of $\K_v$}) is a partition of the edges of $\K_v$ into {\it perfect matchings}
\footnote{A perfect matching of a graph $\Gamma$ is a subset of $E(\Gamma)$ partitioning $V(\Gamma)$.}.
The $q$-analog of a perfect matching of $\K_v$ is clearly a parallel class of the point-line design associated with PG$(\F_q^v)$, 
i.e., a line spread of PG$(\F_q^v)$. It is then natural to give the following definition.
\begin{defn}
A resolvable 2$-(v,2,1)_q$ design is a partition of the lines of PG$(\F_q^v)$ into classes each
of which is a line spread of PG$(\F_q^v)$.
\end{defn}


Adopting this terminology, a famous result by R.D. Baker \cite{Baker} can be restated as follows.
\begin{thm} {\rm\cite{Baker}}
There exists a resolvable $2$$-(v,2,1)_2$ design if and only $v$ is even.
\end{thm}

A near resolvable 2$-(v,2,1)$ design is a partition of the edges of $\K_v$ into {\it near perfect matchings}\footnote{A near
perfect matching of a graph $\Gamma$ is a subset of $E(\Gamma)$ partitioning $V(\Gamma)\setminus\{x\}$ for a suitable
{\it missing vertex} $x$.}. 
Obviously, a near perfect matching $N$ of $\K_v$ can be seen as a perfect matching of $\K_{v-1}$.
Thus, for what we have said above, the $q$-analog of $N$ should be a 1-spread of 
a PG$(\F_q^{v-1})$, i.e., a 1-spread of a hyperplane of PG$(\F_q^v)$. 
It is then natural to give the following definition.
\begin{defn}
A near resolvable 2$-(v,2,1)_q$ design is a partition of the lines of PG$(\F_q^v)$ into classes each
of which forms a spread of a hyperplane.
\end{defn}

Several authors \cite{FJV,FMT,Tonchev} studied the problem of partitioning a Singer $([v]_q,[v-1]_q,[v-2]_q)$ difference set 
into lines of PG$(\F_q^v)$ forming a $([v]_q,[2]_q,1)$ difference family. 
Every solution of this problem clearly gives a cyclic near resolvable 2$-(v,2,1)_q$ design.

In particular, every such solution is a Singer graceful labeling of the graph $\Gamma$ of order $[v-1]_q$ whose 
connected components are $(q+1)$-cliques, hence it can be viewed as a 2$-(v,\Gamma,1)$ design over $\F_q$.
The converse is not true; we may have a Singer graceful labeling of $\Gamma$ that is not a partition into lines.
Here is an example. Take a root $g$ of the polynomial $x^5+x^2+1$ as generator of $[\Z_5]_2$, take the 
Singer $(31,15,7)$ difference set $D=\{1,2,3,4,6,8,12,15,16,17,23,24,27,29,30\}$ considered in subsection \ref{singergracefulcycles}, 
and let $\Gamma$ be the graph whose connected components are five $3$-cliques. A $D$-graceful labeling of $\Gamma$ is
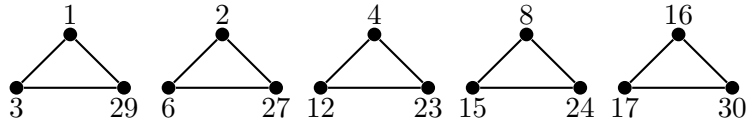
\begin{center}
\begin{tikzpicture}[-,auto,node distance=2cm,thick,main
node/.style={circle,fill=black,draw}]
\node[circle,fill,scale=0.5](1) {};
\node[circle,fill,scale=0.5, below right of=1](2) {};
\node[circle,fill,scale=0.5, below left of=1](3) {};
\node[draw=none,fill=none,scale=0.5, right of= 1](4) {};
\node[circle,fill,scale=0.5, right of=4](5) {};
\node[circle,fill,scale=0.5, below right of=5](6) {};
\node[circle,fill,scale=0.5, below left of=5](7) {};
\node[draw=none,fill=none,scale=0.5, right of= 5](8) {};
\node[circle,fill,scale=0.5, right of=8](9) {};
\node[circle,fill,scale=0.5, below right of=9](10) {};
\node[circle,fill,scale=0.5, below left of=9](11) {};
\node[draw=none,fill=none,scale=0.5, right of= 9](12) {};
\node[circle,fill,scale=0.5, right of= 12](13) {};
\node[circle,fill,scale=0.5, below right of=13](14) {};
\node[circle,fill,scale=0.5, below left of=13](15) {};
\node[draw=none,fill=none,scale=0.5, right of= 13](16) {};
\node[circle,fill,scale=0.5, right of=16](17) {};
\node[circle,fill,scale=0.5, below right of=17](18) {};
\node[circle,fill,scale=0.5, below left of=17](19) {};

\path
(1) edge node {} (2)
(2) edge node {} (3)
(3) edge node {} (1)
(5) edge node {} (6)
(6) edge node {} (7)
(7) edge node {} (5)
(9) edge node {} (10)
(10) edge node {} (11)
(11) edge node {} (9)
(13) edge node {} (14)
(14) edge node {} (15)
(15) edge node {} (13)
(17) edge node {} (18)
(18) edge node {} (19)
(19) edge node {} (17)
;

\node [above] at (1) {1};
\node [below] at (2) {29};
\node [below] at (3) {3};

\node [above] at (5) {2};
\node [below] at (6) {27};
\node [below] at (7) {6};

\node [above] at (9) {4};
\node [below] at (10) {23};
\node [below] at (11) {12};

\node [above] at (13) {8};
\node [below] at (14) {24};
\node [below] at (15) {15};

\node [above] at (17) {16};
\node [below] at (18) {30};
\node [below] at (19) {17};
\end{tikzpicture}
\end{center}

On the other hand none of the cliques of the above graph is a line of PG$(\F_2^5)$.
Indeed the sum of the three elements of the $i$-th clique is $g^i+1$.

\section{Improper graph decompositions over a finite field}

Given a graph $\Gamma$, let us denote by $I(\Gamma)$ the set of its isolated vertices and by $\Gamma\setminus I(\Gamma)$
the graph obtained from $\Gamma$ by deleting all its isolated vertices. The graph obtained with the opposite operation of adding 
a certain number $d$ of isolated vertices to a graph $\Gamma$  will be denoted by $\Gamma \ \cup \ N_d$. Indeed, by $N_d$ we mean the
{\it null graph} of order $d$, i.e., the edgeless graph with $d$ vertices.

Usually, speaking of a 2$-(v,\Gamma,\lambda)$ design, it is understood that $I(\Gamma)$ is empty.
This is because if ${\cal B}$ is the collection of blocks of a 2$-(v,\Gamma,\lambda)$ design, then it is obvious that
$\{B\setminus I(B) \ | \ B\in{\cal B}\}$ is the collection of blocks of a 2$-(v,\Gamma\setminus I(\Gamma),\lambda)$ design.
Conversely, it is clear that from a 2$-(v,\Gamma,\lambda)$ design one can immediately obtain a 2$-(v,\Gamma \ \cup \ N_d,\lambda)$ design
provided that $d\leq v-k$ where $k$ is the order of $\Gamma$.

Anyway, to allow isolated vertices in the context of graph decompositions over a finite field is meaningful.
Indeed, deleting the isolated vertices of each block of a 2$-(v,\Gamma,\lambda)$ design over $\F_q$
we do not obtain a 2$-(v,\Gamma\setminus I(\Gamma),\lambda)$ design over $\F_q$.

We will say that a 2$-(v,\Gamma,\lambda)$ design over $\F_q$ is {\it improper of degree $d$} if $\Gamma$ has exactly $d$ isolated vertices.
An improper design of degree 0 will be said {\it proper}.
That said, it is clear that the most interesting designs are the proper ones. 
Suffice it to say that every possible 2$-([v]_q,\Gamma,\lambda)$ design can be ``extended" to a suitable 2$-(v,\Gamma,\lambda)$ design over $\F_q$;
in the worst of the cases, said $k$ the order of $\Gamma$, it gives a spanning 2$-([v]_q,\Gamma \ \cup \ N_{[v]_q-k},\lambda)$ design over $\F_q$.

As an example, we give a 2$-(7,\Pi_3 \ \cup \ N_1,1)$ design over $\F_2$, so improper of degree 1.
Reasoning exactly as in Subsection \ref{Q_3^*} the reader can check that such a design can be obtained
by means of only one initial base block that is the preimage under $f$ of the graph depicted below.
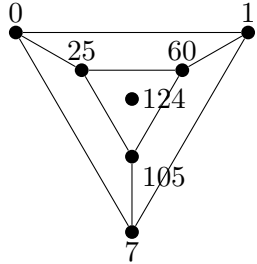
\begin{center}
\begin{tikzpicture}
\node[minimum size=1.5cm, regular polygon, regular polygon sides=3,
rotate=180] (epta) {};
\foreach \x in {1,2,3}{%
\node[circle,fill,scale=0.5] at (epta.corner \x) (e\x) {};
}
\node[minimum size=3.5cm, regular polygon, regular polygon sides=3,
rotate=180] (septa) {};
\foreach \x in {1,2,3}{%
\node[circle,fill,scale=0.5] at (septa.corner \x) (s\x) {};
}
\node[circle,fill,scale=0.5](1) {};
\path
(e1) edge node {} (e2)
(e2) edge node {} (e3)
(e3) edge node {} (e1)
(s1) edge node {} (s2)
(s2) edge node {} (s3)
(s3) edge node {} (s1)
(e1) edge node {} (s1)
(e2) edge node {} (s2)
(e3) edge node {} (s3)
;
\node [below right] at (e1) {105};
\node [above] at (e2) {60};
\node [above] at (e3) {25};
\node [below] at (s1) {7};
\node [above] at (s2) {1};
\node [above] at (s3) {0};
\node [right] at (1) {124};
\end{tikzpicture}
\end{center}

\section*{Conclusion}
It would be nice to conclude with a list of open problems. The fact is that almost everything is still open. Even though the
topic of designs over finite fields has been considerably relaxed, to find general answers seems to be extremely difficult.
At the moment, for the case $\lambda=1$, we are not even able to exhibit an infinite family of non-trivial $\Gamma$-decompositions over a finite field.
So a natural target, hopefully not too ambitious, should be to prove that there are infinitely many values of $v$ for which there exists a 
2$-(v,\Gamma,1)$ design over $\F_q$ for at least one pair $(q,\Gamma)$.

\section*{Acknowledgements}
The authors are grateful to the anonymous referees for their comments which improved the readability of the paper.
This work has been performed under the auspices of the G.N.S.A.G.A. of the C.N.R. (National Research Council) of Italy. The second author is supported in part by the Croatian Science Foundation under the project 6732.

\bibliographystyle{model1-num-names}

\end{document}